%% file: main.tex
\documentclass[14pt]{article}
\usepackage[utf8]{inputenc}
\usepackage{pdfpages}
\usepackage{url}
\usepackage{euler} 
\usepackage[toc,page]{appendix}
\usepackage{amssymb, amsthm,amsmath}
\usepackage[Conny]{fncychap}
\usepackage{graphicx}
\usepackage{adjustbox}
\usepackage{comment}
\usepackage{etex}
\usepackage{mathrsfs}
\usepackage[english]{babel}

\usepackage[scaled=0.85]{beramono}%% mono
\usepackage{tikz-cd}
\usepackage{float}
\usepackage{fancyhdr}

\usepackage[left=0cm,right=0cm,top=2cm,bottom=2cm,margin=3cm]{geometry}
\usepackage{amsthm}
\usepackage{bigints}
\usepackage{tikz}
\usepackage[linktoc=all]{hyperref}
\usepackage[capitalise]{cleveref}
\hypersetup{colorlinks=true,linkcolor=blue,citecolor=red}
\usetikzlibrary{matrix,arrows,decorations.pathmorphing}
\usetikzlibrary{patterns,calc,angles,quotes}
\usepackage{tikz-cd}
\usepackage[utf8]{inputenc}
\tikzset{commutative diagrams/.cd}
\newtheorem{theorem}{Theorem}[subsection]

\newtheorem{lemma}[theorem]{Lemma}

\newtheorem{proposition}[theorem]{Proposition}

\theoremstyle{definition}
\newtheorem{definition}[theorem]{Definition}

\newtheorem{claim}[theorem]{Claim}

\newtheorem{example}[theorem]{Example}

\newtheorem{remark}[theorem]{Remark}

\newtheorem{construction}[theorem]{Construction}

\newtheorem{notation}[theorem]{Notation}

\newcommand{\E}{\mathcal{E}}
\newcommand{\D}{\mathcal{D}}

\newcommand{\Q}{\mathbf{Q}}
\newcommand{\Sc}{\op{Sch}}

\newcommand{\K}{\mathcal{K}}

\newcommand{\N}{\mathcal{N}}

\newcommand{\op}[1]{\operatorname{#1}}

\newcommand{\C}{\mathcal{C}}

\newcommand{\U}{\mathcal{U}}

\newcommand{\F}{\mathcal{F}}

\newcommand{\niI}{\underline n}
\newcommand{\ep}{\epsilon}
\newcommand{\dd}{\delta}

\newcommand{\Ca}{\mathcal{C}}

\numberwithin{subsection}{section}

\newcommand{\Set}{\op{Sets}}
\newcommand{\J}{\mathcal{J}}

\newcommand{\sset}{\op{Set}_{\Delta}}

\newcommand{\Cpt}{\op{Cpt}}
\newcommand{\Copt}{\op{\C pt}}
\renewcommand{\U}{\mathcal{U}}
\newcommand{\bx}{\square}
\newcommand{\Crt}{\op{Cart}}
\newcommand{\Cart}{\op{\C art}}
\newcommand{\Krt}{\op{\K art}}
\newcommand{\Kpt}{\op{\K pt}}

\fancypagestyle{main}{\fancyhf{}
	\fancyhead[LE,RO]{\scriptsize the $\infty$-categorical compactification}
	\fancyhead[RE,LO]{\scriptsize \rightmark }
	\fancyfoot[CE,CO]{\thepage}}

\title{ Six-Functor Formalisms II : The $\infty$-categorical compactification.}
\author{Chirantan Chowdhury}
\date{\today}

\begin{document}

\maketitle{}

\begin{abstract}
  This paper is part of a series of articles in which we reproduce the statements regarding the abstract six-functor formalism developed by Liu-Zheng. In this paper, we prove a theorem, which is an $\infty$-categorical version for defining the exceptional pushforward functor in an abstract-six functor formalism. The article describes specific combinatorial simplicial sets related to compactifications and pullback squares. This theorem plays a key role in constructing the abstract six-functor formalism, which will be discussed in the forthcoming article.  
\end{abstract}
\tableofcontents

\input{introduction}
\input{multisimplicialsets}
\input{functorsusingcatofsimplices}

\input{extensionalongpcomm}
\input{extensionalongpcart}

\input{appendices}

\end{document}

%% file: introduction.tex
\section{Introduction}
	
The abstract six-functor formalism plays is an $\infty$-categorical formulation of classical Grothendieck six-functor formalism encoding duality theories \footnote{In these articles, we only consider the abstract six-functor formalism using the language of $\infty$-categories due to Lurie. Abstract six-functor formalisms have been studied using the language of derivators, for example: \cite{hörmann2022derivator}}. In recent years, such formalism became a key tool in formalizing cohomology and duality theories in various contexts like arithmetic geometry (\cite{padic6functorlucasmann},\cite{Etalecohodiam}), motivic homotopy theory (\cite{khan2021generalized},\cite{Chowdhury}) many others. The foundational ideas of abstract six-functor formalism rely on unpublished works due to Liu-Zheng (\cite{Gluerestnerv} and \cite{liu2017enhanced}). This article is part of a series where we reprove a simplified version of statements proved by Liu-Zheng in constructing abstract six-functor formalism. In particular, this article deals with proving an $\infty$-categorical version of constructing exceptional pushforwards by gluing functors using combinatorial simplicial sets and the technical theorem proved in the previous article (\cite[Theorem 4.1.1]{chowdhury2023sixfunctorformalismsi}).\\

To motivate the abstract formalism of compactifications, let us consider the construction of exceptional functors in the setting of \'etale cohomology of schemes.
	Let $f: X \to Y$ be a separated morphism of finite type of quasi-compact and quasi-separated schemes and $\Lambda$ be a torsion ring. We have the extraordinary pushforward map on the level of triangulated categories \[ f_!: \D(X,\Lambda) \to \D(Y,\Lambda)  \]
	which when restricted to open immersions is the map $f_{\#}$ and to proper morphisms the map $f_*$. The construction of $f_!$ involves the general theory of gluing two psuedofunctors developed by Deligne (\cite[Section 3]{Delignecohomologyproper}). Let us briefly recall the construction setup. 
 \begin{definition}\label{compactification2catdef}
      For any morphism $f$ as above, we consider the $2$-category of compactifications $\op{Sch}^{\op{comp}} $ whose objects are schemes and morphisms are triangles
	\begin{equation}
		\begin{tikzcd}
			{} & \overline{Y} \arrow[dr,"p"] & {} \\
			X \arrow[ur,hookrightarrow,"j"] && Y
		\end{tikzcd}
	\end{equation}
	where $j$ is open and $p$ is proper.
 \end{definition}
 It is important to note that one can compose morphisms of such form due to Nagata's compactification theorem. Then one can define a pseudo-functor $F_c: \op{Sch}^{\op{comp}} \to \op{Cat}_1$ which sends a scheme $X$ to $D(X,\Lambda)$ and a triangle of the form above to the composition $p_* \circ j_{\#}$ (here $\op{Cat}_1$ denotes the $2$-category of categories). The theory of gluing in $2$-categories tells us that the functor $F_c$ can be extended to a functor $f_!$ from the category $\op{Sch}'$ consisting of schemes where morphisms are separated and finite type. In other words, the diagram 
	\begin{equation}
		\begin{tikzcd}
			\op{Sch}^{\op{comp}} \arrow[r,"F_c"] \arrow[d,"\op{pr}"] & \op{Cat}_1 \\
			\op{Sch}' \arrow[ur,"F_!"] & {}
		\end{tikzcd}
	\end{equation} 
	
	In the language of abstract $\infty$-categories, we reinterpret the following statement in the following fashion: Let $\Ca$ be an $\infty$-category and let $(\E_1,\E_2)$ be a pair of collection of edges in $\Ca$ satisfying some nice conditions (see \cref{compthm} for more details). Let us consider a new simplicial set $\dd^*_2 \Ca^{\op{cart}}_{\E_1,\E_2}$. The $n$-simplices of $\dd^*_2\Ca^{\op{cart}}_{\E_1,\E_2}$ are $n \times n$ grids of the form 
	
	\begin{equation}
		\begin{tikzcd}
			X_{00} \arrow[r] \arrow[d] & X_{01} \arrow[r] \arrow[d] & \cdots & X_{0n} \arrow[d] \\
			X_{10} \arrow[r]  & X_{11} \arrow[r] & \cdots & X_{1n} \\
			\vdots & \vdots & \vdots & \vdots \\
			X_{n1} \arrow[r] & X_{n2} & \cdots & X_{nn}
		\end{tikzcd}
	\end{equation}
	where vertical arrows are in $\E_1$, horizontal arrows are in $\E_2$ and each square is a pullback square. Also one has a natural morphism $p:\dd^*_2\Ca^{\op{cart}}_{\E_1,\E_2} \to \Ca$ induced by composition along the diagonal. Before stating the main theorem, let us recall the notion of admissible edges.
	\begin{definition}
		Let $\Ca$ be an $\infty$-category. Let $\E$ be a collection of morphisms in $\Ca$. Then $\E$ is said to be \textit{admissible} if 
		\begin{enumerate}
			\item $\E$ contains every identity morphism in $\Ca$.
			\item $\E$ is stable under pullbacks.
			\item For every pair of composable morphisms $p \in \E$ an $q$ a morphism in $\Ca$, then if $p \circ q \in E $ implies $ q \in E$. 
		\end{enumerate}
	\end{definition}
		\begin{theorem}\label{compthm}
		Let $\Ca$ be an $\infty$-category and $\E_1,\E_2$ be a collection of edges in $\Ca$ with the following conditions:
  \begin{enumerate}
      \item For every morphism in $f \in \Ca$, there exists a $2$-simplex in $\Ca$ of the form :
      \begin{equation}
          \begin{tikzcd}
              {} & y \arrow[dr,"p"] & {}\\
              x \arrow[ur,"q"] \arrow[rr,"f"] && z
              \end{tikzcd}
              \end{equation}
              where $p \in \E_1$ and $q \in \E_2$.
\item Every morphism $f \in \E_1\cap \E_2$ is $k$-truncated for $k \ge -2$.
 
\item The edges $\E_1$ and $\E_2$ are admissible.

  \end{enumerate}
  Then for any $\infty$-category $\D$, there exists a solution to the lifting problem:
\begin{equation}
\begin{tikzcd}
    \dd^*_2\Ca^{\op{cart}}_{\E_1,\E_2} \arrow[d,"p"] \arrow[r,"g"] &\D\\
    \Ca \arrow[ur,"g'",dotted] & {}.
    \end{tikzcd}
\end{equation}
  
	\end{theorem}

Notice that the morphism $p$ admits the following factorization :
\begin{equation}
    \dd^*_2\Ca^{\op{cart}}_{\E_1,\E_2} \xrightarrow{p_{\op{cart}}} \dd^*_2\Ca_{\E_1,\E_2} \xrightarrow{p_{\op{comm}}} \to \Ca. 
\end{equation}
where the middle simplicial set is defined in a similar way as $\dd^*_2\Ca^{\op{cart}}_{\E_1,\E_2}$. This simplicial set comprises $n$-simplices, which are $n \times n$ grids as above but with commutative squares (not necessarily pullback squares). 

The proof of \cref{compthm} follows from extending the morphism $g$ along $p_{\op{cart}}$ and $p_{\op{comm}}$. Thus, the above theorem follows from the following two theorems:

\begin{theorem}[Theorem A : Extension along $p_{\op{comm}}$]\label{thmA}
	Let $\Ca$ be an $\infty$-category and $\E_1,\E_2$ be a collection of edges in $\Ca$ with the following conditions:
  \begin{enumerate}
      \item For every morphism in $f \in \Ca$, there exists a $2$-simplex in $\Ca$ of the form :
      \begin{equation}
          \begin{tikzcd}
              {} & y \arrow[dr,"p"] & {}\\
              x \arrow[ur,"q"] \arrow[rr,"f"] && z
              \end{tikzcd}
              \end{equation}
              where $p \in \E_1$ and $q \in \E_2$.

\item The edges $\E_1$ and $\E_2$ are admissible.
 \end{enumerate}
  Then for any $\infty$-category $\D$, there exists a solution to the lifting problem:
\begin{equation}
\begin{tikzcd}
    \dd^*_2\Ca_{\E_1,\E_2} \arrow[d,"p_{\op{comm}}",swap] \arrow[r,"g_{\op{comm}}"] &\D\\
    \Ca \arrow[ur,"g'_{\op{comm}}",swap,dotted] & {}.
    \end{tikzcd}
\end{equation}
  \end{theorem}

  \begin{theorem}[Theorem B : Extension along $p_{\op{cart}}$]\label{thmB}
    		Let $\Ca$ be an $\infty$-category and $\E_1,\E_2$ be a collection of edges in $\Ca$ with the following conditions:
  \begin{enumerate}
\item Every morphism $f \in \E_1\cap \E_2$ is $k$-truncated for $k \ge -2$.
\item The edges $\E_1$ and $\E_2$ are admissible.

  \end{enumerate}
  Then for any $\infty$-category $\D$, there exists a solution to the lifting problem:
\begin{equation}
\begin{tikzcd}
    \dd^*_2\Ca^{\op{cart}}_{\E_1,\E_2} \arrow[d,"p_{\op{cart}}",swap] \arrow[r,"g_{\op{cart}}"] &\D\\
    \dd^*_2\Ca_{\E_1,\E_2} \arrow[ur,"g'_{\op{cart}}",swap,dotted] & {}.
    \end{tikzcd}
\end{equation}
    
  \end{theorem}
In the first article (\cite{chowdhury2023sixfunctorformalismsi}), we reprove a technical theorem that states conditions for solving lifting problems appearing in the above theorems. We shall use \cite[Theorem 4.1.1]{chowdhury2023sixfunctorformalismsi} to prove both theorems. \\ To apply the above theorem, we first need to study specific simplicial sets that encode compactifications and cartesian squares. Part of this article fairly relies on understanding these specific simplicial sets and their properties.\\

We briefly outline the sections of the article :
\begin{enumerate}
    \item In Section 1, we recall the notion of multi-simplicial sets and their variants, which encode markings and tilings. These notions provide a convenient way of understanding the simplicial sets in the setting of abstract six-functor formalism.
    '\item In Section 2, we recall the relevant definition and the main theorem from \cite{chowdhury2023sixfunctorformalismsi}. We also lay out a road map on how we approach using this theorem in proving \cref{thmA} and \cref{thmB}
    \item In Section 3, we prove \cref{thmA}. A majority part of this section involves the combinatorics understanding the \textit{$\infty$-category of compactifications} (\cref{inftycatofcompactificationsdef}), which is a combinatorial way to encode compactifications for $n$-composable morphisms. Proving relevant properties of this simplicial set leads us to prove \cref{thmA}.
    \item In Section 4, we prove \cref{thmB}. Analog to the previous section, here we study the \textit{$\infty$-category of Cartesianizations} (\cref{inftycatofcartesianizationsdef}). This is a combinatorial way to decompose commutative squares into pullback squares. We prove the theorem by producing analogous results as in the previous section, hence completing the proof of the \cref{compthm}.
    \item In Appendix A, we discuss the combinatorial properties of partially ordered sets, which are relevant to the specific simplicial sets considered in Sections 3 and 4.
    \item In appendix B, we discuss properties of the existence of limits in overcategories. These technical statements are needed in proving the weak contractibility of simplicial sets in Sections 3 and 4, which is a key point in applying \cref{maintechnicalsimplicesthm}.
    \item In Appendix C, we recall the notion of $k$-truncated morphisms and their properties.
    \end{enumerate}

\subsection*{ Acknowledgements:}

The paper was written while the author was a PostDoc under Prof.Dr. Timo Richarz at University of TU Darmstadt. C.Chowdhury acknowledges support (through Timo Richarz) by the European Research Council (ERC) under Horizon Europe (grant agreement nº 101040935), by the Deutsche Forschungsgemeinschaft (DFG, German Research Foundation) TRR 326 \textit{Geometry and Arithmetic of Uniformized Structures}, project number 444845124 and the LOEWE professorship in Algebra, project number LOEWE/4b//519/05/01.002(0004)/87. \\

The author would like to thank Alessandro D'Angelo and R{\i}zacan \c{C}ilo\u{g}lu for helpful discussions regarding the paper. 

\subsection*{Conventions:}

The paper relies on notations and definitions from the paper \cite{Gluerestnerv}. We shall omit referencing the paper as it will be implicit throughout the article. We also freely use the language of $\infty$-categories developed by Lurie in \cite{HTT}, \cite{HA} and \cite{SAG}. 

%% file: multisimplicialsets.tex
\section{Multisimplicial, multi-marked and multi-tiled simplicial sets.}
	\subsection{Multisimplicial sets.}
	Let $I$ be a finite set and consider it as a discrete category.
	\begin{definition}
		An \textit{I-simplicial set} is a functor:
		\[ \op{Fun}(I,\Delta)^{op}=  (\underbrace{\Delta \times \Delta  \cdots \Delta}_{\text{I-times}})^{op} \to \Set  \]
		We denote the category of $I$-simplicial sets by $\Set_{I\Delta}$. If $I=\{1,2,\cdots,k\}$, then we denote it by $\op{Set}_{k\Delta}$.
	\end{definition}
	\begin{remark}
		By definition, $\Set_{1\Delta}=\Set_{\Delta}$ and similarly $\Set_{2\Delta}$ is the category of bisimplicial sets.
	\end{remark}
	\begin{notation}
		We shall denote any object $(n_i)_{i \in I} $ of $\op{Fun}(I,\Delta)$ by $\niI$
		We denote $\Delta^{\niI}$ to be the $I$-simplicial set represented by $\prod_i\Delta^{n_i}$. For an $I$-simplicial set, we denote $S_{\niI}$ by $S(\niI)$. 
	\end{notation}
	
	We discuss adjunctions between $\Set_{k\Delta}$ and $\Set_{\Delta}$. 
	
	\begin{notation}
		\begin{enumerate}
			\item Denote $f: I=\{1,2,\cdots,k\}\to \{1\}$ be the projection map. This induces the functor $\Delta \to \op{Fun}(I,\Delta)$ which induces the \textit{diagonal functor}:
			\[ \delta^*_k:\Set_{k\Delta} \to \Set_{\Delta}  \] 
			which takes an $k$-simplicial set $S$ to $\delta^*_k(S)$ which evaluated on $[n]$ is $ S([n],[n],\cdots [n])$. \\
			This functor has a right adjoint:
			\[\delta^k_*:\Set_{\Delta} \to \Set_{k\Delta}  \]
			which evaluated on $S$, defines a $k$-simplicial set defined as 
			\[ \delta^k_*(S)_{\niI} = \op{Hom}_{\Set_{\Delta}}(\prod_{i \in I} \Delta^{n_i},S) \]
			\item  Similarly an injection of sets $f: J \hookrightarrow I$ induces a functor $(\Delta_f)^*: \Set_{J\Delta} \to \Set_{I\Delta}$ induced from $f$. It has a right adjoint, which we denote by 
			\[ \ep^I_J: \Set_{I\Delta} \to \Set_{J\Delta} \]
			defined by 
			\[ \ep^I_J(S)(\niI) = S((\niI,0)) \]
			where  we write $(\niI,0)$ for the vector with entries $0$ for $i \neq j$. We call the map $\ep^I_J$ as \textit{restriction functor}. \\
			If $I=\{1,2,\cdots,k\}$ and $J =\{j\}$, then we denote it by $\ep^k_j$. 
			\item Given $I=\{1,2,\cdots,k\}$ and $J \subset I$. We have the \textit{partial opposite} functor \[\op{op}^I_J: \Set_{k\Delta} \to \Set_{k\Delta} \] defined by taking opposite edges along the directions $j \in J$.  Using this notion, we define the \textit{twisted diagonal } functor as  \[ \dd^*_{k,J}:= \dd^*_I \circ \op{op}^I_J: \Set_{k\Delta} \to \Set_{\Delta}. \]
		\end{enumerate}
	\end{notation}
	\begin{example}
		\begin{enumerate}
			\item 	The map $\dd^2_*$ takes a simplicial set $S$ to the bisimplicial set $\dd^2_*S$ whose $(n_1,n_2)$ simplices are $\op{Hom}_{\sset}(\Delta^{n_1} \times \Delta^{n_2},S)$. If $S = \Ca$ where $\Ca$ is an ordinary category, then these are just $n_1 \times n_2$ grids in $\Ca$. \\
			The map $\dd^*_2$ takes a bisimplicial set to its diagonal simplicial set. For $S =\dd^2_*N(C)$, the $n$-simplices of the simplicial set  $\dd^*_2(\dd^2_*\Ca$ are morphisms $\Delta^n \times \Delta^n \to \Ca$ (in other words these are $n \times n$ grids in $C$).\\
			\item 	The maps $\ep^2_1$ and $\ep^2_2$ send a bisimplicial set $S':(\Delta \times \Delta)^{op} \to \Set $ to the simplicial sets $S'|_{\Delta \times [0]}$ and $S'|_{[0] \times \Delta}$ respectively, i.e.. these are the restrictions to the first row and column of the bisimplicial set.
			
			\item For $k=1$, the twisted diagonal functor sends a simplicial set $S$ to $S^{op}$.For $k=2$, the partial opposite functor $\op{op}^2_{\{1\}}$ takes a bisimplicial set $S$ and sends to the bisimplicial set $S'$ which when restricted to direction $1$ gives the simplicial set $(\ep^2_1 S)^{op}$ and when restricted to direction $2$ gives the simplicial set $\ep^2_2S$.  In order to understand it more clearly, let us consider the bisimplicial set $\dd^2_*\Ca$. Then the $n$ simplices of the simplicial set $\dd^*_{2,\{1\}}(\dd^2_{*}\Ca)$ are given by $n \times n$ grids $(\Delta^n)^{op} \times \Delta^n \to \Ca$.
			
		\end{enumerate}
	\end{example}
	
	\subsection{Multi-marked and multi-tiled simplicial sets.}
	
	\begin{definition}
		An \textit{$I$-marked simplicial set} is the data $(S,\E:=\{\E_i\}_{i \in I})$ where $S$ is a simplicial set and $\E$ is a set of edges $\E_i$ containing every degenerate edge of $S$. A morphism between $I$-marked simplicial sets $(S,\E)$ and $ (S',\E')$ is a morphism of simplicial sets $ f: X \to X'$ with the property $f(\E_i) \subset \E'_i$. We denote the category of $I$-marked simplicial sets as $\Set^{I+}_{\Delta}$. If $I = \{1,2,\cdots, k\}$, we denote the category of $I$-marked simplicial sets by $\Set^{k+}_{\Delta}$. 
	\end{definition}
	\begin{remark}
		An $I$-marked simplicial set is said to be an $I$-marked $\infty$-category if the underlying simplicial set is an $\infty$-category. \\
		For $k=1$, we get the notion of marked simplicial sets defined in \cite[Section 3.1]{HA}. 
	\end{remark}
	
	\begin{notation}
		
		\begin{enumerate}
			\item Given any $I$-simplicial set $S$, we can define an $I$-marked simplicial set $\dd^*_{I+}(S') = (\dd^*_I S', \E = \{ (\epsilon^I_iS)_1\}_{i \in I})$. When $k=2$, the marked simplicial set $\dd^*_{+2}(S')$ consists of the diagonal simplicial set of $S'$ with the marked edges being the edges of the simplicial set of the first row and first column of the bisimplicial set.
			\item Given any $I$-marked simplicial set $(S,\E)$, we can define an $I$-simplicial set $\dd^{I+}_*(S,\E)$ as the sub $I$-simplicial set of $\dd^I_*S$ which consists only of edges $\E_i$ in simplicial set $\ep^I_i(S)$. 
		\end{enumerate}
		
	\end{notation}
	This notion yields us to define the notion of restricted simplicial nerve. 
	\begin{definition}
		Let $(S,\E)$ be an $I$-marked simplicial set, then we define the \textit{restricted $I$-simplicial nerve} as 
		\[ S_{\E}:= \dd^{I+}_* (S,\E) \]
	\end{definition}
	\begin{example}
		Let $(S,\E) = (N(\Sc), \{P,O\})$ where $P$ and $O$ are the set of proper morphisms and open immersions respectively. Then $S_{\E}$ is the bisimplicial subset of the bisimplicial set $\dd^2_*N(\Sc)$ which consists of only proper morphisms as edges in the simplicial set $\ep^2_1(\dd^2_*N(\Sc))$ and open immersions as edges in the simplicial set $\ep^2_2(\dd^2_*N(\Sc))$.
	\end{example}
	\begin{definition}
		An \textit{I-tiled simplicial set} is the data $(S,\E=\{\E_i\}_{i\in I},\Q=\{\Q_{ij}\}_{i,j \in I, i \neq j})$ where $(X,\E)$ is a marked simplicial set and $\Q$ is a collection of set of squares $\Q_{ij}$ (i.e. $\Delta^1 \times \Delta^1 \to S$) such that 
		\begin{enumerate}
			\item the set of squares $\Q_{ij}$ and $\Q_{ji}$ are obtained from each other by transposition.
			\item The vertical arrows of each square in $\Q_{ij}$ are in $\E_i$ and the horizontal arrows are in $\E_j$.
			\item To every edge in $\E_i$, there is a square in $\Q_{ij}$ induced by the map $\op{id} \times s^0_0$. 
		\end{enumerate} A morphism of $I$-tiled simplicial sets $ f: (S,\E,\Q) \to (S',\E',\Q')$ which maps $f(\E_i) \subset \E'_i$ and $f(\Q_{ij}) \subset \Q'_{ij}$. We denote the category of $I$-tiled simplicial sets by $\Set^{I\bx}_{\Delta}$.
	\end{definition}
	
	\begin{notation}
		
		\begin{enumerate}
			\item Given any $I$-simplicial set $S$, we define an $I$-tiled simplicial set $\dd^*_{I\bx} (S):= (\dd^*_I S, \E,\Q)$ where $\E=\{\E_i= (\epsilon^I_i S')_1\}_{i\in I}$ and $\Q = \{\Q_{ij}= \op{Hom}(\Delta^1 \times \Delta^1,\dd^*_2\epsilon^I_{i,j}(S))\}_{i,j \in I, i \neq j}$. 
			\item Given any $I$-tiled simplicial set $(S',\E',\Q')$, we define an $I$-simplicial set $\dd^{I\bx}_*((S',\E',\Q'))$ as the  $I$-simplicial subset of $\dd^{I+}_*(S',\E')$ such that for $j,k \in I$ and $j \neq k$, every square in the simplicial set $\dd^*_2\ep^I_{jk}(\dd^{I+}_*(S',\E'))$ associated to any $(1,1)$-simplex lies in $\Q_{jk}$.
		\end{enumerate}
	\end{notation}
	\begin{remark}
	  Let $S$ be a bisimplicial set. Given any $(1,1)$-simplex of $S$, we can define a square in the diagonal simplicial set $\dd^*_2S$ as follows. A $(1,1)$-simplex corresponds to a morphism $\tau: \Delta^{(1,1)} \to S$. Applying the functor $\dd^*_2 (-)$, we get a morphism \[ \dd^*_2(\tau): \Delta^1 \times \Delta^1 \to \dd^*_2S. \] \\
		If $S = N(\Sc')_{P,O}$, then a square in $S$ corresponds to a morphism $ \Delta^1 \times \Delta^1 \to N(\Sc')$ where horizontal arrows are proper and vertical arrows are open. 
	\end{remark}
	\begin{definition}
		Let $\Ca$ be and $\infty$-category and $\E_1,\E_2$ be set of edges, denote $\E_1 \star^{\op{cart}} \E_2$ be the set of Cartesian squares. For an $I$-marked $\infty$-category $(\Ca,\E:=\{\E_i\}_{i \in I})$, we denote $\E^c_{ij}:= \E_i \star^{\op{cart}} \E_j$. Denote $ (\Ca,\E,\E^c)$ to be the $I$-tiled $\infty$-category. We define the \textit{Cartesian $I$-simplicial nerve} to be the $I$-simplicial set 
		\[ \Ca^{\op{cart}}_{\E}:= \dd^{I\bx}_*((\Ca,\E,\E^c)) \] 
	\end{definition}
	
	\begin{example}
		The bisimplicial set $N(\Sc)^{\op{cart}}_{P,O}$ is the sub-bisimplicial set of $\dd^2_*N(\Sc)$ which consists of proper morphisms as edges in one direction, open immersions as edges in other and every square formed by open and proper morphisms is a pullback square. \\
		
		Let us understand the simplicial set $\dd^*_{k}\Ca^{\op{cart}}_{\E}$ which will be the source of the enhanced operation map. A $n$-simplex of  $\dd^*_k\Ca^{\op{cart}}_{\E}$ is a morphism $\sigma_n:\underbrace{\Delta^n \times \Delta^n \cdots \times \Delta^n}_{ \op{k-times}} \to \Ca$  such that every edge  $\sigma_n|_i: \Delta^1 \to \Ca$ in direction $i$ lies in $\E_i$ for every $ i \in I$ and for every $ j \neq j' \in I$, the square $\sigma_n|_{j,j'}: \Delta^1 \times \Delta^1 \to \Ca$ is a pullback square formed by edges $\E_j$ and $\E_{j'}$. In case $ k=2$ and $(C,\E) = (\Sc,\{P,O\})$, the $n$-simplices of $\dd^*_2N(\op{Sch}')^{cart}_{P,O}$ are $n \times n$ grids of the form 
		
		\begin{equation}
			\begin{tikzcd}
				X_{00} \arrow[r] \arrow[d] & X_{01} \arrow[r] \arrow[d] & \cdots & X_{0n} \arrow[d] \\
				X_{10} \arrow[r]  & X_{11} \arrow[r] & \cdots & X_{1n} \\
				\vdots & \vdots & \vdots & \vdots \\
				X_{n1} \arrow[r] & X_{n2} & \cdots & X_{nn}
			\end{tikzcd}
		\end{equation}
		where vertical arrows are proper, horizontal arrows are open and each square is a pullback square. 
	\end{example}
	

%% file: functorsusingcatofsimplices.tex
\section{Recollection of key results from \cite{chowdhury2023sixfunctorformalismsi}. }
In this section, we recall the relevant notions and \cite[Theoerem 4.1.1]{chowdhury2023sixfunctorformalismsi}from \cite{chowdhury2023sixfunctorformalismsi}.

\subsection{The global section functor.}
\begin{definition}
    Let $\J$ be a (small) ordinary category. Let $(\sset)^{\J}$ be the category where objects are functors from $\J \to \sset$ and morphisms are natural transformations.
\end{definition}
We now introduce the constant and global section functor related to $(\sset)^{\J}$.
\begin{notation}
    For every simplicial set $X$, we have the constant simplicial set functor $c(X):=X_{\J}$ defined by sending any object $j$ to the simplicial set $X$. The association is functorial and thus we have a functor :
     \[ c: \sset \to (\sset)^{\J}\]
    
\end{notation}

\begin{definition}\label{globalsectionfunctor}
We define the \textit{global section functor} 
\[ \Gamma : (\sset)^{\J} \to \sset\]
as follows : 
\[ \Gamma(F) =(\Gamma(F)_n := \op{Hom}_{(\sset)^{\J})}(\Delta^n_{\J},F))_n.\]
\end{definition}
\begin{example}
Let $F$ be the constant functor $c(X)$ where $X \in \sset$. Let us compute $\Gamma(F)$.
The $n$-simplices of $\Gamma(F)$ are given by the set of natural transformations from $\Delta^n_J \to c(X)$. Every such natural transformation is equivalent to give a single map $\Delta^n \to X$. In particular the $n$-simplices of $\Gamma(F)$ are given by $n$-simplices of $X$.Thus $\Gamma(F)= X$. In particular, we prove that $\Gamma \circ c = \op{id}_{\sset}$.

\end{example}
\begin{remark}
Recall from classical category theory, given a complete category $\Ca$ and an small category $I$, we have the pair of adjoint functors: 
 \[ c: \Ca \leftrightarrows \op{Fun}(I,\Ca) : \op{lim}\]
 where $\op{lim}$ is the functor which takes an object which is a functor $F : I \to \Ca$ to its limit $\op{lim}(F) \in \Ca$. \\

 Let $\Ca= \sset$ and $I=\J$. As the category of simplicial sets is complete,  we see that 
 \[ \Gamma = \op{lim}. \]
 In the other words, the global section functor is the limit functor which takes every functor to its limit in the category of simplicial sets.
 
\end{remark}
\subsection{Category of simplices.}
\begin{definition}\label{catofsimpldef}
    Let $K$ be a simplicial set. Then the \textit{category of simplicies over $K$} is a category consisting of :
    \begin{enumerate}
        \item Objects : $(n,\sigma)$ where $n \ge 0$ and $\sigma \in K_n$.
        \item Morphisms: $ p:(n,\sigma)\to (m,\sigma')$ is a morphism $p: [n] \to [m]$ such that $p(\sigma)=\sigma'$.
    \end{enumerate}
\end{definition}

The relevant functor associated to the category of simplices is the mapping functor. 
\begin{definition}\label{mappingfunctor}
    Let $K$ be a simplicial set and $\Ca$ be a $\infty$-category. The \textit{mapping functor} \[ \op{Map}[K,\Ca] : (\Delta_{/K})^{op} \to \sset \] is defined as follows: 
    \[ (n,\sigma) \to \op{Map}^{\sharp}((\Delta^n)^{\flat},\Ca^{\natural}) \cong \op{Fun}^{\cong}(\Delta^n,\Ca). \]
     Here $\op{Map}^{\sharp}((\Delta^n)^{\flat},\Ca^{\natural})$ is the internal mapping space in the category of marked simplicial sets. It is the largest Kan complex contained in $\op{Fun}(\Delta^n,\Ca)$.
 \end{definition}

 \begin{remark}

       For a simplicial set $K$  and an $\infty$-category $\Ca$, we have the following equality of simplicial sets: 
     \[ \Gamma(\op{Map}[K,\Ca]) = \op{Map}^{\sharp}(K^{\flat},\Ca). \] This is proved in \cite[Lemma 3.3.3]{chowdhury2023sixfunctorformalismsi}

 \end{remark}

 \subsection{The main theorem and road map for proving \cref{thmA} and \cref{thmB}.}
 We recall the main theorem from \cite{chowdhury2023sixfunctorformalismsi}. 
 \begin{theorem}\label{maintechnicalsimplicesthm}\cite[Theorem 4.1.1]{chowdhury2023sixfunctorformalismsi}
Let $K',K$ be simplicial sets and $\Ca$ be a $\infty$-category. Let $f': K' \to \Ca$ and $i:K' \to K$ be morphisms of simplicial sets. Let $\N \in (\sset)^{(\Delta_{/K})^{op}}$ and $\alpha: \N \to \op{Map}[K,\Ca]$ be a natural transformation. If
\begin{enumerate}
 \item (\textit{Weakly contractibility}) for $(n,\sigma) \in \Delta_{/K}$, $\N(n,\sigma)$ is weakly contractible,
 \item (\textit{Compatability with $f'$}) there exists $\omega \in \Gamma(i^*\N)_0$ such that $\Gamma(i^*\alpha)(\omega)=f'$,
 \end{enumerate}
 
then there exists a map $f: K \to \Ca$ such that the following diagram 
\begin{equation}
    \begin{tikzcd}
        K' \arrow[r,"f'"] \arrow[d,"i"] & \Ca \\
        K \arrow[ur,"f"] & {}
    \end{tikzcd}
\end{equation}
 commutes. In other words, $f' \cong f \circ i $ in $\op{Fun}(K',\Ca)$.
\end{theorem}

\subsection*{Road map for proving \cref{thmA} and \cref{thmB}.}\label{roadmap}

The overall arching goal is to use \cref{maintechnicalsimplicesthm} for proving both of these theorems. We use the notations from the theorem to explain the brief sketch.

\begin{itemize}
    \item (\textbf{Weakly contractibility:}) In our cases $K$ and $K'$ will be the simplicial sets of the form  $\dd^*_2\Ca_{\E_1,\E_2}$. To any $n$-simplex $\tau$ in the target $K'$, we try to associate a simplicial set $\N(\tau)$ which relates to the target. \\
   In the case of \cref{thmA}, we define using the \textit{$\infty$-category of compactifications} (\cref{inftycatofcompactificationsdef}) which encodes the way of decomposing a $n$-simplex into specific directions. We show that this $\infty$-category is weakly contractible (\cref{kptweaklycontractible}).\\
   In the case of \cref{thmB}, we define as $\N(\tau)$ using the notion of  \textit{$\infty$-category of cartesianizations} (\cref{inftycatofcartesianizationsdef}). This encodes the way of decomposing commutative squares into cartesian squares. It follows that such a category is a weakly contractible Kan complex(\cref{KrtcontKancomplex}). 
   \item (\textbf{Construction of $\alpha$:}) This is really involved in both of the theorems. In both of these cases, the morphism $\alpha$ uses the map $f'$. Secondly, it uses some inner andoyne properties of specific combinatorial simplicial sets.

   In the case of \cref{thmA}, the key morphism is \cref{alphamapforpcom}. The inner anodyne morphism property in this section is proved in \cref{bxncptinneranodyne}.\\
   In thee case of \cref{thmB}, the simplicial sets get more technical. The construction of $\alpha$ follows from construction of $\epsilon_n$ (\cref{epsilonconstruction}) and \cref{truncatedCartalphaconstruction} . Also the analog inner anodyne property is proved in \cref{cartninneranodyne}.

   \item(\textbf{Compatibility with $f'$:}) In both of these theorems, the construction of $\alpha$ really uses the existence of the map $f'$. This enables us to prove the compatibility condition of \cref{maintechnicalsimplicesthm}.
\end{itemize}

%% file: extensionalongpcomm.tex
\section{Extension along $p_{\op{comm}}$}
In this section, we try to solve the lifting problem of the following form :
\begin{equation}
\begin{tikzcd}
    \dd^*_2\Ca_{\E_1,\E_2} \arrow[d,"p_{\op{comm}}",swap] \arrow[r,"g_{\op{comm}}"] &\D\\
    \Ca \arrow[ur,"g'_{\op{comm}}",swap,dotted] & {}.
    \end{tikzcd}
\end{equation}

As both $\dd^*_2\Ca_{\E_1,\E_2}$ and $\Ca$ have same objects, we know what $g'_{\op{comm}}$ is on the level of objects. Let us try to analyze how to define $g'_{\op{comm}}$ for morphisms. 
\begin{itemize}
    \item Let $f : x \to y$ be a morphism in $\Ca$. Then by conditions in \cref{thmA}, we know that $f$ admits a decomposition of the form :
     \begin{equation}
          \begin{tikzcd}
              {} & y \arrow[dr,"p"] & {}\\
              x \arrow[ur,"q"] \arrow[rr,"f"] && z
              \end{tikzcd}
              \end{equation}
              where $p \in \E_1$ and $q \in \E_2$.
    \item Consider the following diagram :
    \begin{equation}
        \begin{tikzcd}
            x \arrow[r,"q"] \arrow[d,"\op{id}"] & y \arrow[d,"\op{id}"] & {} \\
            x \arrow[r,"q"]  & y \arrow[d,"p"] \arrow[r,"\op{id}"] & y \arrow[d,"p"] \\
            {} & z \arrow[r,"\op{io}"] & z. 
        \end{tikzcd}
    \end{equation}
    This diagram defines a morphism :
    \begin{equation}
        \Lambda^2_1 \to \dd^*_2\Ca_{\E_1,\E_2} 
    \end{equation}
    Notice that this square does not necessarily fill to a $2$-simplex hinting that $\dd^*_2\Ca_{\E_1,\E_2}$ is not an $\infty$-category.
    \item Composing it with $g_{\op{comm}}$, we have the solution of the following lifting problem: 
    \begin{equation}
        \begin{tikzcd}
         \Lambda^2_1 \arrow[r,"h'"] \arrow[d,hookrightarrow] & \D \\
         \Delta^2 \arrow[ur,dotted,"h",swap] & {}
        \end{tikzcd}
    \end{equation}
    We "define" \[g'_{\op{comm}}(f) = h(\{0 \to 2 \})\] 
\end{itemize}
Note that such a construction above depends on the chosen decomposition of $f$ hence the " " sign on the defintion. We now list the major ideas in proving \cref{thmA} :
\begin{enumerate}
    \item The key simplicial set that we need to encode all such possible decompositions for higher simplicies is the simplicial set |$\Copt^n$. For $n=1$, the simplicial set is $\Delta^2$. We shall also introduce a subsimplicial set $\bx^n \subset \Copt^n$ which shall encodes the directions of these decompositions in $\E_1$ and $\E_2$ respectively. For $n=1$ $\bx^1$ turns out to be the inner horn $\Lambda^2_1$. Using technical simpicial arguments of partially ordered sets, we shall show that $\bx^n \to \Copt^n$ is an inner anodyne. This plays a key role in constructing the map $\alpha$ in \cref{maintechnicalsimplicesthm}.
    \item Given any $n$-simplex of $\Ca$, we shall define the \textit{$\infty$-category of compactifications of $\tau$} (denoted by $\Kpt(\tau)$. This shall encode all various ways of decomposiing an $n$-simplex of $\tau$ in directions of $\E_1$ and $\E_2$. We shall show that the collection of such decompositions is weakly contractible (\cref{kptweaklycontractible}).
    \item In order to construct the morphism $\alpha$,  we shall define the morphism : 
    \[\alpha_{\op{comm}} : \Kpt^n \to \op{Fun}(\bx^n,\dd^*_2\Ca_{\E_1,\E_2}) \]
    which is a simplicial way of encoding the procedure in $n=1$ case mentioned above.
    \item Combining all the above points, we shall conclude the proof of \cref{thmA} verifying the conditions of \cref{maintechnicalsimplicesthm}.
\end{enumerate}

\subsection{The $\infty$-category of compactifications.}
In this subsection, we define two important simplicial sets: simplicial set of compactifications and cartesianizations. The simplicial set of compactifications is an important tool for showing that the $p_{\op{com}}$ is a categorical equivalence. The simplicial set of cartesianizations is need for showing that the $p_{\op{cart}}$ is a categorical equivalence. The definitions of both of these are motivated from the ideas of proving the theorem. \\
	We shall define specific simplicial sets which shall play a key role in defining these objects (see \cite[Section 4]{Gluerestnerv} and \cite[Section 5]{Gluerestnerv} are main references for the notations). \\
	
	\begin{definition}
	 Let $\Kpt^n$ be the sub-bisimplicial set of the bisimplicial set $\Delta^{n,n}$ spanned by vertices $(i,j)$ where $1 \le i \le j \le n$. 
	\end{definition}
	\begin{definition}
		Let $\Cpt^n \subset [n] \times [n]$ be the category spanned by objects $(i,j), 1 \le i \le j \le n$. Denote $\Copt^n:= N(\Cpt^n)$. 
	\end{definition}
	
	\begin{notation}
		Denote $\bx^n:=\dd^*_2 \Kpt^n$.  Also for a partially ordered set $P$ with ordering $\le$ and two elements $x,y \in P$, we denote:
		\begin{enumerate}
			\item $P_{x/}$ to be the undercategory of $x$. 
			\item $P_{/x}$ to be the overcategory of $x$. 
			\item $P_{x//y}$ to be the category spanned by objects $z \in P$ where $x \le z \le y$. It is empty if $x  > y $.
		\end{enumerate}
	\end{notation}
	
	\begin{remark}
		Some remarks on the definitions above. 
		\begin{enumerate}
  \item The simplicial set $\bx^n$ also admits the following description : 
  \begin{equation}
      \bx^n  \cong \bigcup_{i=0}^n N(\Cpt^n_{(0,i)//(i,n)}).
  \end{equation}
			\item Diagram of $\Copt^1$ is as follows:
			\begin{center}
				\begin{tikzcd}
					a_{00} \arrow[r] & a_{01} \arrow[d]\\
					{} & a_{11}
				\end{tikzcd}
			\end{center}
			Thus $\Copt^1 \cong \Delta^2$. Here $a_{ij}$ is the vertex $(i,j)$ in $[n] \times [n]$.
			\item We have a natural inclusion $\bx^n \subset \Copt^n$. 
			\item The Hasse diagram of $\bx^1$ is as follows: 
			\[ \bx^1:= \begin{tikzcd}
				a_{00} \arrow[r] & a_{01}
			\end{tikzcd} 
			\cup \begin{tikzcd}
				a_{01} \arrow[d] \\
				a_{11}
			\end{tikzcd} \]
			Thus $\bx^1 \cong \Lambda^2_1$. 
			\item The Hasse diagram of $\Copt^2$ is as follows:
			
			\begin{center}
				\begin{tikzcd}
					a_{00} \arrow[r] & a_{01} \arrow[r] \arrow[d] & a_{02} \arrow[d] \\
					{} & a_{11} \arrow[r] & a_{12} \arrow[d] \\
					{} & {} & a_{22}
				\end{tikzcd}
			\end{center}
			Thus $\Copt^2 \cong \Delta^4 \coprod_{\Delta^1} \Delta^4$. 
			\item The Hasse diagram of $\bx^2$ is as follows:
			
			\[\bx^2 \cong \begin{tikzcd}
				a_{00} \arrow[r] & a_{01} \arrow[r] & a_{02}
			\end{tikzcd} \cup \begin{tikzcd}
				a_{01} \arrow[r] \arrow[d] & a_{02} \arrow[d]\\a_{11} \arrow[r] & a_{12}
			\end{tikzcd} \cup \begin{tikzcd}
				a_{02} \arrow[d] \\
				a_{12} \arrow[d] \\
				a_{22}
			\end{tikzcd}\]
		
		\end{enumerate}
	\end{remark}
	\begin{proposition}\label{bxncptinneranodyne}
	
		The inclusion $\bx^n \subset \Copt^n$ is inner anodyne. 
	
	\end{proposition}
 \begin{proof}
    Let $P = \Cpt^n$ and let $p_1 = (0,0),... p_{n+1}= (0,n)$ and $q_1 = (n,0),\cdots q_{n+1}= (n,n)$. Then applying \cref{partiallyorderedmainprop}, we get that $\bx^n \subset \Copt^n$ is an inner anodyne.
 \end{proof}
The category $\op{Cpt}^n$ admits a nice stratification as described below:
 \begin{notation}
     For every $0 \le j \le n$ and for $n \ge 1$, let $\op{Cpt}^n_0 = \op{Cpt}^{n-1} \cup \{(n,0)\}$ and by induction we let $\op{Cpt}^{n-1}_j:= \op{Cpt}^{n-1}_{j-1} \cup \{(n,n-j)\}$. Let $\Copt^n_j$ be the corresponding nerves of the partially ordered sets. In particular, we have a sequence of simplicial sets:
     \begin{equation}
         \Copt^{n-1} \subset \Copt^n_0 \subset \Copt^n_1 \cdots \Copt^n_{n-1} \subset \Copt^n_n= \Copt^n.
     \end{equation}
 \end{notation}
	We now define the simplicial set of compactifications.
	
	\begin{definition}\label{inftycatofcompactificationsdef}
	Let $\Ca$ be an $\infty$-category with $\E_1,\E_2$	be a pair of admissible edges in $\Ca$. Let $\tau: \Delta^n \to \Ca$, then the \textit{ $\infty$-category of compactifications of $\tau$}  is a  subcategory of $\op{Fun}(\Cpt^n,\Ca)$ spanned by objects $\tau':\Cpt^n \to \Ca$ such that :
 \begin{enumerate}
     \item $\tau'|_{\Delta^n} =\tau$.
     \item $\tau'$ sends arrows $(i,j) \to(i,j+1)$ to an edge in $\E_2$.
     \item $\tau'$ sends arrows $(i,j) \to (i+1,j)$ to an edge in $\E_1$.
      \item Morphism between two objects are give by natural transformations which pointwise gives an edge in $\E_1$.
    \end{enumerate}
    The $\infty$-category shall be denoted as $\Kpt(\tau)$
	
	\end{definition}

The following proposition says that the collection of compactifications is weakly contractible. This plays a key important role in proving Theorem A.
	\begin{proposition}\label{kptweaklycontractible}
	Let $\Ca$ be an $\infty$-category with  $\E_1,\E_2$ a pair of edges satisfying the following conditions :
      \begin{enumerate}
      \item For every morphism in $f \in \Ca$, there exists a $2$-simplex in $\Ca$ of the form :
      \begin{equation}
          \begin{tikzcd}
              {} & y \arrow[dr,"p"] & {}\\
              x \arrow[ur,"q"] \arrow[rr,"f"] && z
              \end{tikzcd}
              \end{equation}
              where $p \in \E_1$ and $q \in \E_2$.
\item The edges $\E_1$ and $\E_2$ are admissible.
 \end{enumerate}	
  Then the $\infty$-category $\Kpt(\tau)$ is weakly contractible. 
	\end{proposition}

 The proof for general $n$-simplex $\tau$, one needs to show that for when $\tau$ is a $1$-simplex, the following proposition holds.

 \begin{lemma}\label{Kpt1weakcontractible}
     Let $\tau$ be an edge in $\Ca$. Then the $\infty$-category $\op{Kpt}(\tau)$ is cofiltered hence weakly contractible.
 \end{lemma}
 \begin{remark}
     \textbf{Evidence of the lemma:} In order to verify the cofiltered condition, we need to show for the very basic case that given any two compactifications of a morphism $f : x \to y$ given by $\tau_1 : x \to z \to y$ and $ \tau_2 : x \to z' \to y$ , we need to find a third compactification $ \tau_3 : x \to z'' \to y$ and morphisms $\tau_3 \to \tau_2$ and $\tau_3 \to \tau_1$ respectively. Let us explain how this goes : 
     \begin{itemize}
         \item Consider the diagram : 
         \begin{equation}
             \begin{tikzcd}
                 x \arrow[r] z \arrow[d] & z \arrow[d] \\
                 z' \arrow[r] & y
             \end{tikzcd}
         \end{equation}
         \item As $\Ca_{\E_1}$ admits pullbacks, decompose the diagram into the following : 
         \begin{equation}
             \begin{tikzcd}
                 x \arrow[dr,"f'"] & {} & {} \\
                 {} &  x' \arrow[r] \arrow[d] & z \arrow[d] \\
                 {} & z' \arrow[r] & y
             \end{tikzcd}
         \end{equation}
         where the inner square is a pullback square and the edges of this square are in $\E_1$.
         \item Note that the morphism $f'$ may not be in $\E_2$, but we can decompose $f'$ into a composition $x \to z'' \to x'$ where $x \to z'$ is in $\E_2$ and $z \to x'$ is in $\E_1$.
         \item Adding this decomposition, we get the following diagram :
         \begin{equation}
             \begin{tikzcd}
                 x \arrow[dr] \arrow[ddr] \arrow[drr] & {} & {} \\
                 {} & z'' \arrow[r] \arrow[d] & z \arrow[d] \\
                 {} & z' \arrow[r] & y
             \end{tikzcd}
            \end{equation}
            where the inner square consists of all edges in $\E_1$. 
            \item Thus we have defined a third compactification $\tau_3 : x \to z'' \to y$ and two mmorphism of compactifications : $\tau_3 \to \tau_1$ and$\tau_3 \to \tau_2$
     \end{itemize}
     The general case of cofiltered condition follows the similar idea but generalized so it works for general $n$-simplices. 
 \end{remark}
 \begin{proof}[Proof of \cref{Kpt1weakcontractible}]
     Note that $\Cpt^1 \cong \Delta^2$. For $m \ge 1$, we need to show the solution of the lifting problem :
     \begin{equation}
         \begin{tikzcd}
             \partial\Delta^m \arrow[r,"g_m"] \arrow[d, hookrightarrow] & \op{Kpt}(\tau) \\
             \Lambda^{m+1}_0 \arrow[ur,dotted,"g_m'"] & {}
         \end{tikzcd}
         \end{equation}
         Unravelling the definition, we need to solve the following lifting problem :
         \begin{equation}
             \begin{tikzcd}
                 \partial\Delta^m \times \Delta^2 \arrow[r,"g_m"] \arrow[d,hookrightarrow] & \Ca \\
                 \Lambda^{m+1}_0 \times \Delta^2 \arrow[ur,dotted,"g_m'"] & {}
             \end{tikzcd}
      \end{equation}
      where $g_m'$ can be realized as a morphism $\Lambda^{m+1}_0 \to \op{Kpt}(\tau)$. This lifting problem reduces to construct a new object of $\op{Kpt}(\tau)$ which shall be the image of the initial vertex of $\Lambda^{m+1}_0$  with a coherent filling of other subsimplices to get the desired map $g_m'$.\\ 
      
    Let $\tau : x \to y$ be the edge in $\Ca$.   Note that the map $g_m$  when restricted to every vertex of $\partial \Delta^m$ maps the vertex $[0]$ and $[2]$ to $x$ and $y$ respectively. As $\E_2$ is an admissible set of edges, we can consider the morphism $g_m$ valued in the overcategory $\Ca_{/y}$ :
    \begin{equation}
        h_m : \Delta^0 \boldsymbol{* }\partial\Delta^m \to \Ca_{/y}
    \end{equation}
    According to the conditions in the proposition, we have that $\Ca_{\E_2}$ admits pullbacks and $\Ca_{\E_2} \to \Ca$ preserves pullbacks. By \cref{finitelimitsinovercategories}, we see that $\Ca_{\E_2,/y}$ admits finite limits and the morphism $\Ca_{\E_2,/y} \to \Ca_{/y}$ preserves finite limits. Thus the morphism $h_m|_{\partial\Delta^m}$ admits a limit in $\Ca_{/y}$. Thus we get a morphism
    \begin{equation}
        h_m' : \partial\Delta^1\boldsymbol{*} \partial\Delta^m \to \Ca_{/y}
    \end{equation}
    Let us denote $z'$ be the image of cone point along the morphism $h_m'$.
    By universal property of limit diagram, we get the morphism $h_m'$ extends to 
    \begin{equation}
        h_m'' : \Delta^1 \boldsymbol{*} \partial\Delta^m \to \Ca_{/y}
    \end{equation}
    The morphism $h_m''|_{[0]}$ amounts to defining an edge $x \to z'$ in $\Ca$. In particular we have defined a new $2$-simplex 
    \begin{equation}
        \begin{tikzcd}
            {} & z' \arrow[dr,"q'"] & {}\\
            x \arrow[ur,"p'"]  \arrow[rr,"\tau"]&& y
        \end{tikzcd}
        \end{equation}
        where $q \in \E_2$ but $p$ may not belong to $\E_1$.  In order to define $g_m'$, we need to get a simplex where $p' \in \E_1$. But we know that $p'$ admits a decomposition $\sigma:$
        \begin{equation}
            \begin{tikzcd}
                     {} & z \arrow[dr,"q''"] & {}\\
            x \arrow[ur,"p"]  \arrow[rr,"p'"]&& z'
            \end{tikzcd}
         \end{equation}
                 where $p \in \E_1, q'' \in E_2$.  The amalgamation of $h_m''$ amd $\sigma$ amounts to the existence of the morphism :
                 \begin{equation}
                     \widetilde{h_m} : \Delta^1 \boldsymbol{*} \partial\Delta^m \coprod_{\Delta^1 } \Delta^2 \to \Ca_{/y}
                 \end{equation}
                 where $\Delta^1 \hookrightarrow \Delta^2$ is given by edge $0\to 2$. This inclusion is right anodyne as it is composition of $\Delta^1 \hookrightarrow \Lambda^2_2 \hookrightarrow \Delta^2$ which are both right anodynes. Thus by \cite[Lemma 2.1.2.3]{HTT}, we see that $\widetilde{h_m}$ extends to a morphism :
                 \begin{equation}
                     g_m' : \Delta^2 \boldsymbol{*} \partial\Delta^m \cong \Delta^1 \boldsymbol{*}\Lambda^{m+1}_0 \to \Ca_{/y} \to \Ca
                  \end{equation}
                  where $g_m'|_{[0]}$ is the morphism $x \to z$ which is an edge in $\E_1$. Thus we can realize $g_m' : \Lambda^{m+1}_0 \to \Kpt(\tau)$ as the conditions of the edges are fulfilled. This gives us the desired extension and completing the proof.  

 \end{proof}
 
 Let us now proceed in proving the proposition.
 \begin{proof}[Proving \cref{kptweaklycontractible}]

 The goal is to show that $\Kpt(\tau)$ is cofiltered which implies weakly contractibility. In particular, we show that for all $m \ge 0$, we show that the lifting problem :
 \begin{equation}\label{filteredcompliftingproblem}
     \begin{tikzcd}
         \partial\Delta^m \arrow[r,"f_m"] \arrow[d, hookrightarrow] & \Kpt(\tau)\\
         \Lambda^{m+1}_{0} \arrow[ur,dotted,"f_m'"] & {}
          \end{tikzcd}
 \end{equation}
 admits a solution.  As $\op{Kpt}^n(\Ca)$ is a subcategory of $\op{Fun}(\Copt^n,\Ca)$, it boils down to lift a similar diagram :
 \begin{equation}
     \begin{tikzcd}
         \partial \Delta^m \times \Copt^n \arrow[r,"f_m"] \arrow[d,hookrightarrow] & \Ca \\
         \Lambda^{m+1}_0 \times \Copt^n \arrow[ur,"f_m'"] & {}
     \end{tikzcd}
 \end{equation}

 We proceed to prove this statement by induction on $n$.
 \begin{enumerate}
     \item \textbf{n=0:} In this case, there is nothing to prove. 
     \item \textbf{n=k-1 $\implies$ n=k:}.  We assume that we have the lifting property for $n=k-1$, as noticed in the notation above, we have the decomposition :
      \begin{equation}
         \Copt^{n-1} \subset \Copt^n_0 \subset \Copt^n_1 \cdots \Copt^n_{n-1} \subset \Copt^n_n= \Copt^n.
     \end{equation}. We show that we can lifting the problem levelwise on $\Copt^n_j$ by induction. In precise, we prove the following claim.
     \begin{claim}
     Let us assume that $f_m$ can be lifted to a map 
     \begin{equation}
         f_m^{j-1} : \Lambda^{m+1}_0 \to\op{Fun}^{\E_1,\E_2}(\Copt^n_{j-1},\Ca).
     \end{equation}
     where the latter category is the full subcategory  of $\op{Fun}(\Copt^n_{j-1},\Ca)$ spanned by functors which when restricted to $\Delta^n$ is the simplex $\tau$, sends horizontal arrows to $\E_1$ and vertical arrows to $\E_2$. Then we have a solution to the lifting problem :
     \begin{equation}\label{liftingproblemoftheclaim}
     \begin{tikzcd}
         \partial\Delta^m \times \Copt^n_j \coprod_{\partial\Delta^n \times \Copt^n_{j-1}} \Lambda^{m+1}_{0} \times \Copt^n_{j-1} \arrow[d,hookrightarrow] \arrow[r," \alpha_j"] & \Ca\\
         \Lambda^{m+1}_{0} \times \Copt^n_j \arrow[ur,"f^j_m",dotted,swap] & {}
         \end{tikzcd}
     \end{equation}
    such that arrow $f^j_m$ induces a morphism $\Lambda^{m+1}_0 \to \op{Fun}^{\E_1,\E_2}(\Copt^n_j,\Ca)$. Here $\alpha=(f_m|_{\Copt^n_j}, f_m^{j-1})$
     \end{claim}

\textbf{Claim implying the proposition:} Once we prove the claim proving it for $j=n$, it implies the existence of the lifting problem for \cref{filteredcompliftingproblem}. Thus we are reduced to proving the claim.
\begin{proof}[\textbf{Proof of the claim}]

\begin{itemize}
    \item \textbf{j=0:} Note that $\tau \in( (\Delta^n  \to \Ca)$, The amalgamation of $\tau$ and  extension of $f_m' : \Lambda^{m+1}_0 \times \Copt^{n-1} \to \Ca$ gives the map 
    \begin{equation}
        \tilde{f_m} : \Lambda^{m+1}_0 \times (\Copt^{n-1} \coprod_{\Delta^{n-1}} \Delta^n) \to \Ca. 
    \end{equation}
The \cref{filteredcompliftingproblem} now boils down to solving the following lifting problem
\begin{equation}
    \begin{tikzcd}
        \partial\Delta^m \times \Copt^n_0 \coprod_{ \partial\Delta^m \times (\Copt^{n-1} \coprod_{\Delta^{n-1}} \Delta^n) } \Lambda^{m+1}_{m+1} \times (\Copt^{n-1} \coprod_{\Delta^{n-1}} \Delta^n) \arrow[d,hookrightarrow] \arrow[r] & \Ca\\
        \Lambda^{m+1}_{0} \times \Copt^n_0 \arrow[ur,dotted, "f^0_m"] & \Ca.
    \end{tikzcd}
\end{equation}
By \cite[Corollary 2.3.2.4]{HTT}, the solution exists if $\Copt^{n-1} \coprod_{\Delta^{n-1}} \Delta^n \hookrightarrow \Copt^n_0$ is inner anodyne. This follows from \cref{inneranodynepartiallyorderedsetcond} applying $P =\Delta^{n-1}, Q= \Copt^{n-1}, R = \Delta^n$. The condition that $f^m_0$ induces a map $\Lambda^{m+1}_0 \to \op{Fun}(\Copt^n_0 ,\Ca)$ as the addition of the point $(n,n)$ does not add any vertical or horizontal edges in the Hasse diagram of $\Copt^n_0$. This completes the proof for $j=0$.
\item $\textbf{j-1 $\implies$ j:}$ Assume the existence of $f^{j-1}_m$, Notice that $\Cpt^n_j := \Cpt^n_{j-1}\coprod_{[1]}[2]$ where $[1] \to [2]$ is given by $0 \to 0, 1 \to 2$. Then we have an obvious inclusion
\begin{equation}
    \beta_j :  X_j:=\Copt^n_{j-1} \coprod_{\Delta^1}\Delta^2 \hookrightarrow \Copt^n_j
\end{equation}
Using the arguments in the previous point, by \cref{inneranodynepartiallyorderedsetcond}, we see that $\beta_j$ is an inner anodyne.
For the moment, let us assume the following \textbf{assumption}: \\
\begin{center}\label{assumptionforinductionclaim}
 $f^m_{j-1}$ extends to a map $\tilde{f}^m_j : \Lambda^{m+1}_0 \times X_j \to \Ca$ which can be realized as a map $\Lambda^{m+1}_0 \to \op{Fun}^{\E_1,\E_2}(X_j, \Ca)$.
\end{center}
By the assumption, we are again reduced to solve the lifting problem:
\begin{equation}
\begin{tikzcd}
    \partial\Delta^m \times \Copt^n_j \coprod_{\partial\Delta^m \times X_j} \Lambda^{m+1}_0 \times X_j \arrow[d,hookrightarrow] \arrow[r] & \Ca\\
    \Lambda^{m+1}_0 \times \Copt^n_j \arrow[ur,dotted,"f^m_j",swap] & {}
    \end{tikzcd}
\end{equation}
As $\beta_j$ is an inner anodyne, by \cite[Corollary 2.3.2.4]{HTT}, we see that $f^m_j$ exists completing the proof of the induction. This it remains to prove the assumption. 

\begin{proof}[\textbf{Proof of the assumption}]

It is enough to show the existence of the solution of the following problem:

\begin{equation}
    \begin{tikzcd}
        \Lambda^{m+1}_0 \times \Delta^1 \coprod_{\partial\Delta^m \times \Delta^1} \partial\Delta^m \times \Delta^2 \arrow[r,"g"]\arrow[d,hookrightarrow] & \Ca\\
        \Lambda^{m+1}_0 \times \Delta^2 \arrow[ur,dotted,"g'"] & {}
    \end{tikzcd}
\end{equation}
    Here we realize the maps $g$ and $g'$ as morphisms in $\op{Fun}^{\E_1,\E_2}(\Delta^2,\Ca)$ (here we identify $\Cpt^1 \cong \Delta^2$). \\
   A key idea that will be used in this argument is the following : for a simplicial set $K$, we realize the left cone $K^{\triangleleft}$ as the simplicial set $K \times \Delta^1$ where we realize $K \times \Delta^0 \cong K \times [0] \hookrightarrow K \times \Delta^1$ as the degenerate simplex corresponding to a point. We solve this lifting problem in the following steps: 
   \begin{enumerate}
    \item  The map $g$ can be written as a map :
       \begin{equation}
           \partial\Delta^m \times \Delta^1 \times \Delta^1 \coprod_{\partial\Delta^m \times \Delta^1} \partial\Delta^m \times \Delta^2 \to \Ca
       \end{equation}
       Here $\partial\Delta^m \times [1] \times \Delta^1 \hookrightarrow \partial\Delta^m \times \Delta^1 \times \Delta^1$ is the inclusion in the coproduct. One can rewrite the map $g$ as 
       \begin{equation}
           \partial\Delta^m \to \op{Fun}(\Delta^1 \times \Delta^1 \coprod_{\Delta^1}\Delta^2,\Ca)
           \end{equation}

   Let $Q:= \Delta^1 \times \Delta^1 \coprod_{\Delta^1}\Delta^2 $
 Notice that $\Lambda^2_2 = \Delta^1 \coprod_{[0]}\Delta^1 \hookrightarrow \Delta^1 \times \Delta^1 \coprod_{\Delta^1} \Delta^2$ where $\Delta^1 \to \Delta^2$ is the edge $[1] \to [2]$. As $\Ca$ admits pullbacks and edges in direction of $[1] \to [2]$ are in $\E_2$, then we have a morphism $\Lambda^2_2 \to \Ca$ admits a left cone which is a pullback square $\Delta^1 \times \Delta^1 \to \Ca$. This enables us to extend $g$ to a morphism $g_1$ (by \cite[Proposition 4.3.2.15]{HTT})
 \begin{equation}
     g_1 : \partial\Delta^m \to \op{Fun}(Q \coprod_{\Lambda^2_2}\Delta^1 \times \Delta^1, \Ca) 
 \end{equation}
Adding inner horns of $2$-simplex and $3$-simplex (as $\Ca$ is an $\infty$-category) followed by using universal property of limits (i.e initial object in the over category of diagram), we see that $g_1$ can be extended to $g_2$
\begin{equation}
    g_2 : \partial\Delta^m \to \op{Fun}(\Delta^2 \times \Delta^1,\Ca)
\end{equation}
where for every $0 \le i \le m$ and $j = \{0,1\}$
the map $g_2|_{i,j} : \Delta^2 \cong \op{Cpt}^1 \to \Ca$ which sends vertical edges to $\E_2$ (this is true as $\E_2$ is stable under pullbacks) and horizontal edges to $\E_1$. Typically horizontal edge which is constructed by the universal property will not necessarily be in $\E_1$. Using the same argument as in the construction of $g_m'$ from $h_m''$ in \cref{Kpt1weakcontractible}, we can assume the horizontal edges lie in $\E_2$.
\item The morphism $g_2$ when restricted to $\Delta^2 \times \{[0]\}$ gives a morphism $g_2|_{[0]} : \op{Kpt}^1(\tau_0)$ where $\tau_0 := g|_{[0] \times \Delta^1}$ (here $[0]$ is the vertex of $\Lambda^{m+1}_0$). By \cref{Kpt1weakcontractible}, we see that $g_2|_{[0]}$ extends to \[g_2|_{[0]}':\Lambda^{m+1}_0 \times \Delta^2= \partial\Delta^m \times \Delta^1 \times \Delta^2 \to \Ca \]
\item The amalgamation of $g_2$ and $g_2|'_{[0]}$, gives us the map
\begin{equation}
    g_3 : \partial\Delta^m \times \Delta^2 \times \Delta^1 \coprod_{\partial\Delta^m \times \Delta^2} \partial\Delta^m \times \Delta^2 \times \Delta^1 \to \Ca
\end{equation}
where the two inclusions in coproduct are given by $[1] \to \Delta^1$ and $[0] \to \Delta^1$. Rewriting this map, we get \[g_3 : \partial\Delta^m \times \Delta^2 \to \op{Fun}(\Lambda^2_1,\Ca) \]
As $\Ca$ is an $\infty$-category  by \cite[Corollary 2.3.2.2]{HTT}, we get the morphism 
\begin{equation}
    g_4 : \partial\Delta^m \times \Delta^2 \times \Delta^2 \to \Ca
\end{equation}
Now we identify the cone $\Lambda^{m+1}_0$ by $\partial\Delta^m \times\Delta^1$, restricting $g_4$ to $\Delta^m \times \Delta^{1,2} \times \Delta^2$, we get the morphism:
\begin{equation}
    g' : \Lambda^{m+1}_0 \times \Delta^2 \cong \partial\Delta^m \times \Delta^1 \times \Delta^2 \xrightarrow{g_4} \Ca.
\end{equation}

which can be realized as a map $g' : \Lambda^{m+1}_0 \to \op{Fun}^{\E_1,\E_2}(\op{Cpt}^1,\Ca)$.
   \end{enumerate}
\end{proof}
\end{itemize}
\end{proof}
\end{enumerate}

 \end{proof}

 \begin{notation}\label{alphamapforpcom}
     For any $n$-simplex $\tau$ of $\Ca$, we have a canonical morphism :
     \begin{equation}
         \alpha_{\op{comm}} : \op{Kpt}(\tau) \to \op{Fun}(\bx^n, \dd^*_2\Ca_{\E_1,\E_2})
     \end{equation}
     defined as follows: 
     \begin{enumerate}
         \item Let $\sigma : \op{Cpt}^n \to \Ca$ be a zero simplex in L.H.S. We restrict it to $\sigma : \bx^n \to \Ca$. Now recall that $\bx^n = \bigcup_{i=0}^n \bx^n_i$. By definition $\bx^n_i = \Delta^{n-i} \times \Delta^i$. For each $i$, $\sigma$ induces a map of $2$-marked simplicial sets 
         \begin{equation}
             (\Delta^{n-i} \times \Delta^i, \op{ver},\op{hor}) \to (\Ca,\E_1,\E_2)
         \end{equation}
     
    For any two objects $a=(a_1,a_2) \le b =(b_1,b_2) \in \Delta^{n-l} \times \Delta^l$, we define the functors : 
    \begin{equation}
        \Lambda(a,b) = (b_1,0) \vee a \wedge b \quad ; \quad \mu(a,b) = (0,b_2) \vee a \wedge b .
    \end{equation}

Given any $n$-simplex $ \gamma  : \Delta^m \to \Delta^{n-i} \times \Delta^i$, we define : 

\[\alpha_i(\sigma)(\gamma) : \Delta^m \times \Delta^m \to \Ca \]
by  for any $(p,q)$
\begin{enumerate}
    \item $ \sigma(\Lambda(\gamma(q),\gamma(p)))$ \text{when} $p \ge q$ ;
    \item $\sigma(\mu(\gamma(p),\gamma(q))$ \text{when} $p \le q$.
\end{enumerate}
One checks that $\Lambda(a,a)=\mu(a,a)=a$ and from the definitions that $\alpha_i(\gamma) \in (\dd^*_2\Ca_{\E_1,\E_2})_m$. Combining together, we get a map :
\begin{equation}
    \alpha_ i(\sigma) : \bx^n_i:= \Delta^{n-i} \times \Delta^i \to  \dd^*_2\Ca_{\E_1,\E_2}.
\end{equation}
Taking union over $i$'s we get
\begin{equation}
    \alpha_{\op{comm}}(\sigma) : \bx^n \to \dd^*_2\Ca_{\E_1,\E_2}.
\end{equation}

\item For a map on higher simplicies, that is mapping $\Delta^m \times \Copt^n \to \Ca$, one defines in the similar way as aboe but taking care into the compactifications that morphisms between compactfiications go in the direction of $\E_1$.
     \end{enumerate}

 \end{notation}
	
\subsection{Proof of Theorem A.}
We restate the theorem from the introduction.
\begin{theorem}[Theorem A : Extension along $p_{\op{comm}}$]\label{thmA}
	Let $\Ca$ be an $\infty$-category and $\E_1,\E_2$ be a collection of edges in $\Ca$ with the following conditions:
  \begin{enumerate}
      \item For every morphism in $f \in \Ca$, there exists a $2$-simplex in $\Ca$ of the form :
      \begin{equation}
          \begin{tikzcd}
              {} & y \arrow[dr,"p"] & {}\\
              x \arrow[ur,"q"] \arrow[rr,"f"] && z
              \end{tikzcd}
              \end{equation}
              where $p \in \E_1$ and $q \in \E_2$.

\item The edges $\E_1$ and $\E_2$ are admissible.
 \end{enumerate}
  Then for any $\infty$-category $\D$, there exists a solution to the lifting problem:
\begin{equation}
\begin{tikzcd}
    \dd^*_2\Ca_{\E_1,\E_2} \arrow[d,"p_{\op{comm}}",swap] \arrow[r,"g_{\op{comm}}"] &\D\\
    \Ca \arrow[ur,"g'_{\op{comm}}",swap,dotted] & {}.
    \end{tikzcd}
\end{equation}
  \end{theorem}

\begin{proof} As explained in the section \cref{roadmap}, we use the technical lemma to prove Theorem A.  
 \begin{enumerate}
\item  Let $\tau$ be a $n$-simplex of $\Ca$. We have the chain of morphisms : 

\begin{equation}
    \alpha'_n : \op{Kpt}(\tau) \xrightarrow{\alpha_{\op{comm}}} \op{Fun}(\bx^n,\dd^*_2\Ca_{\E_1,\E_2}) \xrightarrow{g_{\op{comm}}} \op{Fun}(\bx^n,\D)
\end{equation}
where : $\alpha_{\op{comm}}$ is the morphism explained in \cref{alphamapforpcom}.
Define :
\begin{equation}
    \alpha_n : N(\tau) : \op{Kpt}(\tau)\times_{\op{Fun}(\bx^n,\D)}\op{Fun}(\Copt^n,\D) \to \op{Fun}(\op{Cpt}^n,\D) \to \op{Fun}(\Delta^n,\D)
\end{equation}
  As $\bx^n \to \Copt^n$ is an inner anodyne (\cref{bxncptinneranodyne}), it follows from \cite[Corollary 2.3.2.5]{HTT} that $\op{Fun}(\Cpt^n,\D) \to \op{Fun}(\bx^n,\D)$ is a trivial vibration. This implies $N(\tau) \to \op{Kpt}(\tau)$ is trivial vibration. As $\op{Kpt}(\tau)$ is weakly contractible (\cref{kptweaklycontractible}), we see that $N(\tau)$ is weakly contractible.

\item One notices that any $1$-simplex in $\op{Kpt}(\tau)$ maps to an equivalence in via $\alpha_n$. Let $\tau$ be the $n$-simplex whose vertices are $\tau_1,\tau_2,\cdots \tau_n$. Then $\alpha_n$ sends any object to $n$-simplex of $\D$ whose vertices are $g_{\op{comm}}(\tau_1),g_{\op{comm}}(\tau_2) \cdots g_{\op{comm}}(\tau_n)$. Then $\alpha_n$ maps a $1$-simplex to an morphism : 
\[\Delta^n \times \Delta^1 \to \D \]
such that $\{i\} \times \Delta^1 \to \D$ for all $0 \le i \le n$ is an equivalence. Hence the  above $1$-simplex is an isomorphism. 
This means $\alpha_n$ can be realized as :
\begin{equation}
    \alpha_n : \N(\tau) \to \op{Fun}^{\simeq}(\Delta^n,\D)
\end{equation}
\item As all the above maps are functorial in $n$, we get the map :
\begin{equation}
    \alpha : \N \to \op{Map}[\Ca,\D]
\end{equation}
\item Now suppose $\tau$ arises from a simplex $\tau'$ in $\dd^*_2\Ca_{\E_1,\E_2}$. Then one can choose an element in $\op{Kpt}(\tau)$ to be the the simplex $\tau'$ restricted to $\tau_1: \Cpt^n \hookrightarrow \Delta^n \times \Delta^n$. Also the simplex $\tau'$ by applying $g_{\op{comm}}$ yields a simplex $\tau_2: \Cpt^n \hookrightarrow \Delta^n \times \Delta^n \to \D$. By construction, it follows that $\alpha_n$ sends $(\tau_1,\tau_2)$ to the element $g_{\op{comm}}(\tau')$. \\

This shows that the compatible collections of $(\tau_1,\tau_2)$ provides an element $\omega \in \Gamma(p_{\op{comm}}^*\N)_0$ which via $p_{\op{comm}}^*(\alpha)$ gets mapped to $g_{\op{comm}}$.

\item By \cref{maintechnicalsimplicesthm}, we get the desired extension $g_{\op{comm}}'$.

\end{enumerate}
\end{proof}

%% file: extensionalongpcart.tex
\section{Extension along $p_{\op{cart}}.$}

This section involves proving theorem B, namely to extend the morphism from cartesian grid squares to commutative grid squares. The key idea to give a combinatorial description of understanding how one decomposes commutative squares into cartesian squares. 
In this context, given a map $g_{\op{cart}}: \dd^*_2\Ca^{\op{cart}}_{\E_1,\E_2} \to \D$, we want to define $g_{\op{cart}}'$ such that the diagram
\begin{equation}
    \begin{tikzcd}
        \dd^*_2\Ca^{\op{cart}}_{\E_1,\E_2} \arrow[d,"p_{\op{cart}}"]\arrow[r,"g_{\op{cart}}"] & \D \\
        \dd^*_2\Ca_{\E_1,\E_2} \arrow[ur,"g_{\op{cart}}'",swap] & {}
    \end{tikzcd}
\end{equation}
commutes.

Let $\tau : \Delta^1 \times \Delta^1 \to \Ca$ be a $1$-simplex of $\dd^*_2\Ca_{\E_1,\E_2}$ of the form :
\begin{equation}
    \begin{tikzcd}
        x \arrow[r] \arrow[d] & y \arrow[d]\\
        z \arrow[r] & w.
    \end{tikzcd}
\end{equation}

Decompose the square into the following diagram :

\begin{equation}
    \begin{tikzcd}
        x \arrow[dr] & {} & {}\\
        {} & x' \arrow[d] \arrow[r] & y \arrow[d] \\
        {} &  z \arrow[r] & w
    \end{tikzcd}
\end{equation}
where $x \to x'$ is in $\E_1 \cap \E_2$. If $x \to x'$ is $-1$-truncated i.e. a monomorphism, then we have a diagram of the form :
\begin{equation}\label{cartesiansquaredecomposition}
    \begin{tikzcd}
        x \arrow[r,"\op{id}"] \arrow[d,"\op{id}"] & x \arrow[d] & {} \\
        x \arrow[r] & x' \arrow[r] \arrow[d] & y \arrow[d] \\
        {} & z \arrow[r] & w
    \end{tikzcd}
\end{equation}
where both squares are pullback squares. This gives us a map :
\begin{equation}
    h' : \Lambda^2_1 \to \dd^*_2\Ca^{\op{cart}}_{\E_1,\E_2} \xrightarrow{g_{\op{cart}}} \D
    \end{equation}
which extends to 
\begin{equation}
    h :\Delta^2 \to \D
\end{equation}

Then we define $g'_{\op{cart}} = h'(0 \to 2)$. This gives us a way to construct the map $g_{\op{comm}}$ for a $1$-simplex in the case where $x \to x'$ is $-1$-truncated. 

Motivating the following construction and mimicking the ideas in extension of $p_{\op{comm}}$, we list the key ideas in extending along $p_{\op{cart}}$ :
\begin{enumerate}
    \item The key simplicial set is $\Cart^n$ which plays a similar role compared to $\Copt^n$. It is constructed from considering partially ordered "upward sets" in $\Delta^n \times \Delta^n$. The simplicial set comes with a natural map $\Delta^n \times \Delta^n \to \Cart^n$. For $n=1$, the Hasse diagram of $\Cart^1$ looks like \cref{cartesiansquaredecomposition}.
    \item Given any $n$-simplex of $\dd^*_2 \Ca_{\E_1,\E_2}$, we construct any  \textit{$\infty$-category of cartesianizations of $\tau$} which is denoted by $\Krt(\tau)$. This  encodes a $n$-simplex analogue of square decomposing into such diagram of the form \cref{cartesiansquaredecomposition}. It turns out that collection of such decompositions is a contractible Kan complex.  
    \item In order to keep the $i$-truncations we proceed by induction. In particular for a fixed $i$, we consider the subsimplicial set $\dd^*_2\Ca^i_{\E_1,\E_2}$ of $\dd^*_2\Ca_{\E_1,\E_2}$ spanned by $n \times n$ grids where each square admits a decomposition of the form \cref{cartesiansquaredecomposition} where $x \to x'$ is $i$-truncated.  Note that :
    \begin{itemize}
        \item $\dd^*_2\Ca^{\op{cart}}_{\E_1,\E_2}= \dd^*_2\Ca^{-2}_{\E_1,\E_2}$
        \item $\dd^*_2\Ca_{\E_1,\E_2} = \bigcup_{i\ge -2} \dd^*_2\Ca^i_{\E_1,\E_2}$.
    \end{itemize}

    The extension along $p_{\op{cart}}$ essentially follows if one proves extension along $p^i_{\op{cart}}$ which is the following map
   \begin{equation}
       p^{i}_{\op{cart}} : \dd^*_2\Ca^{-2}_{\E_1,\E_2} \to \dd^*_2\Ca^i_{\E_1,\E_2}
   \end{equation}
      \item  Similar to $\alpha$ construction in commutative case, given any $n$-simplex $\tau$ of $\dd^*_2\Ca^i_{\E_1,\E_2}$ and given an element $\sigma$ of $\Krt(\tau)$, we define a morphism from a combinatorial simplicial set $\boxplus^{\op{cart}}_n$ (analog of $\bx^n$): 
    \begin{equation}
       \epsilon^{\op{cart}}_n : \boxplus^{\op{cart}}_n \to \dd^*_2\Ca^i_{\E_1,\E_2}
    \end{equation}
The simplicial set admits an inclusion of the form $\boxplus^{\op{cart}}_n \to  \op{Cart}^n$ which is an inner anodyne. 
      \item Following the idea sketch of proving the theorem of $p_{\op{comm}}$, given an $n$-simplex $\tau$ of $\dd^*_2\Ca^i_{\E_1,\E_2}$ one constructs via $\beta$ a map : 
      \begin{equation}
          \alpha_n : \Krt(\tau) \to \op{Fun}(\Cart^n, \D) \to \op{Fun}(\Delta^n,\D)
      \end{equation}
      where $\Delta^n \to \Cart^n$ is the diagonal map. Using the technical lemma of \cref{maintechnicalsimplicesthm}, we shall construct the morphism
      \begin{equation}
          g'_{\op{cart}} = \cup g'^{i}_{\op{cart}} : \dd^*_2\Ca_{\E_1,\E_2} \to \D
      \end{equation}
    \end{enumerate}

\subsection{The $\infty$-category of cartesianizations,}

	In order to encode the diagrams which enables us to decompose commutative squares into pullback squares, we define the notion of up-sets which shall lead us defining the $\infty$-category of cartesianizations, an analogue of $\Kpt(\tau)$. We recall the general notion of lattices and more on partially ordered sets in \cref{partiallyorderedsetsdefinitionandprop}. \\
	
	\begin{definition}
		Let $P$ be a partially ordered set. $Q \subset P$ is said to be an \textit{up-set} if for every $q \in Q$ and $ p \ge q $ in $P$ implies $ p \in Q$. We shall denote the category of up-sets of $P$ by $\U(P)$. It is a partially ordered set where the ordering is given by inverse inclusion. 
	\end{definition}
	\begin{notation}
		For any partially ordered set, we denote the products (infima) by $ \wedge$ and coproducts (suprema) by $\vee$.  In $\U(P)$, we have $ Q \wedge Q' = Q \cup Q'$ and $ Q \vee Q' = Q \cap Q'$. \\
		
		There is a canonical order preserving map $\varsigma^P: P \to \U(P)$ defined by $p \to P_{p/}$. 
	\end{notation}
	There are special squares one considers in a partially ordered set, namely exact squares.
	\begin{definition}
		A square in a partially ordered set is an \textit{exact square} if it is both pushout and a pullback square.
	\end{definition}
	We shall state two important property of exact squares :
    \begin{lemma}\label{morphisminUPinexactsquare}
    Every morphisms $ Q \to Q'$ in $\U(P)$ is a composition of finite sequence of exact pullbacks of morphisms $ \sigma^P(x) \to \sigma^P(x)-x$ where $x \in Q -Q'$.
    \end{lemma}
    \begin{proof}
        We consider a finite chain of compositions $Q \to Q_1 \to Q_2 \cdots Q_i=Q'$ where $Q_j = Q_{j+1} \cup \{x_j\}$. Notice that 
        \begin{equation}
        \begin{tikzcd}
            Q_j \arrow[r] \arrow[d] & \sigma^P(x_j) \arrow[d] \\
            Q_{j+1} \arrow[r] & \sigma^P(x_j)-\{x_j\}
            \end{tikzcd}
        \end{equation}
        is both a pullback square by ordering of $\{x_i\}$ and pushout square also by the ordering. 
    \end{proof}
	\begin{lemma}
		Let $\Ca$ be an $\infty$-category and $F: N(\U(P)) \to \Ca$ be a functor. Then if $F$ is a right Kan extension along $\varsigma^P$, it sends exact squares to pullback squares. 
	\end{lemma}
\begin{proof}

Consider the exact square
\begin{equation}
    \begin{tikzcd}
        Q \cup Q' \arrow[r]\arrow[d] & Q \arrow[d]\\
        Q' \arrow[r] & Q \cap Q'
    \end{tikzcd}
\end{equation}

As $F$ is a Kan extension along $\sigma^n$, we  know that $F(Q \cup Q')$ is the limit over $F(\U(P)_{Q\cup Q'/})$. Consider $R:= \sigma^P(P) \cup\{Q,Q',Q\cap Q'\}$. We have $\sigma^P(P) \subset R \subset \U(P)$. By restriction we know that $F|_R$ is a RKE along $F|_{\sigma^n}$. Then applying \cite[Proposition 4.3.2.8]{HTT}, we see that $F$ is RKE along $F|_{R}$. In other words, $F(Q \cup Q')$ is limit over $F(R_{Q\cup Q'/})$. Thus the argument decreases the objects where the limit is taken over.\\

Our aim is to show that $F$ is the limit of the diagram $\Lambda^2_2 \to R$ given by 
\begin{equation}
\begin{tikzcd}
    {} & Q \arrow[d] \\
    Q' \arrow[r] & Q \cap Q'.
    \end{tikzcd}
\end{equation}
As final maps preserve limits (\cite[Proposition 4.1.1.8]{HTT}, it is enough to show the map $\Lambda^2_2 \to R'$ where $R' =R_{Q\cap Q'/}$ is final. By \cite[Theorem 4.1.3.1]{HTT},we need to check that for every $r \in R'$, the fiber $R'_r :=\Lambda^2_2 \times_{R'}R'_{/r}$ is weaky contractible. Notice that 
\begin{equation}
    R' = R'_{Q/} \cup R'_{Q'/}.
\end{equation}
This the square is a pullback square.
We have two cases now to deal with :
\begin{itemize}
    \item  $Q \le r$ or $Q' \le r$, then $R'_r := Q$ or $Q'$. This it is weakly contractible.
    \item $Q\le r$ and $Q \le r'$ , then by the square being a pushout square, we see that $Q \cap Q '\le r$, this implies $R'_r = \Lambda^2_2$ which is also weakly contractible.
\end{itemize}

This completes the proof of the lemma.

\end{proof}

	We move on to defining the main simplicial set $\Crt^n$ which encodes the information of how to construct cartesian squares out of commutative squares.  
	\begin{definition}
		Consider $[n] \times [n]$. We shall denoted the partially ordered set of non-empty up-sets of $[n] \times [n]$ by $\Crt^n$. \\
		We denote  $\varsigma^n:= \varsigma^{[n] \times [n]}: [n] \times [n] \to \Crt^n$ to be the usual map sending $(p,q) \to ([n] \times [n])_{(p,q)/}$.\\
		Let $\Cart^n:= N(\Crt^n)$ and $\varsigma^n: \Delta^n \times \Delta^n \to \Cart^n$ be the map induced from $\varsigma^n$.  
	\end{definition}
	\begin{remark}
		Some remarks on $\Cart^n$ are as follows:
		\begin{enumerate}
			\item The diagram of $\Cart^1$ is as follows:
			\begin{center}
				\begin{tikzcd}
					b_{00} \arrow[dr] & {} & {} \\
					{} & P \arrow[r] \arrow[d] & b_{01} \arrow[d] \\
					{} & b_{10} \arrow[r] & b_{11}
				\end{tikzcd}
			\end{center}
			Here $b_{ij}:= ([1] \times [1])_{(i,j)/}$ and $P = b_{01} \wedge b_{10}$. 
		\end{enumerate} 
	\end{remark}
	
	\begin{definition}\label{inftycatofcartesianizationsdef}
		Let $\Ca$ be an $\infty$-category. Let $\tau: \Delta^n \times \Delta^n \to  \Ca$ be a map. We define the simplicial set $\Krt(\tau)$ which is defined as the pullback of the diagram:
		
		\begin{center}
			\begin{tikzcd}
				{} & \Krt(\tau)_{\op{RKE}} \arrow[d] \\
				\Delta^0 \arrow[r,"\tau"] & \op{Fun}(\Delta^n \times \Delta^n, \Ca) 
			\end{tikzcd}
		\end{center}	
		where $\Krt(\tau)_{\op{RKE}}$ is the sub-simplicial set of $\op{Fun}(\Cart^n,\Ca)$ which are right Kan extensions along $\varsigma^n$. 
	\end{definition}
	
	\begin{proposition}\label{KrtcontKancomplex}
		If $\Ca$ admits pullbacks, the simplicial set $\Krt(\tau)$ is a contractible Kan complex. 
	\end{proposition}
    \begin{proof}
      Let $ f : \Delta^n \times \Delta^n \to \Ca$ be a diagram, then we want to show that $f$ admits a right Kan-extension along $\sigma^n$. Using \cite[Lemma 4.3.2.13]{HTT}, we need to show that the induced map for every $Q\in \Cart^n$ :
      \begin{equation}
          f_Q : [n] \times [n]_{Q/} \to \Ca
      \end{equation}
      admits a limit. As $\{(n,n)\}$ is the final object of $\op{Cart}^n$ and $[n] \times [n]$, we can realize the morphism :
      \begin{equation}
          f_Q : [n] \times [n]_{Q/} \to \Ca_{/f((n,n))}.
      \end{equation}
      As $\Ca$ admits pullbacks, by \cref{finitelimitsinovercategories} implies that $\Ca_{/f(n,n)}$ admits finite limits. Then $f_Q$ admits a limit. Thus every $f$ admits a right Kan extension.\\

      By \cite[Corollary 4.3.2.15]{HTT} and the explanation above,we see that $\op{Kart}(\tau) \to \op{Fun})\Delta^n \times \Delta^n,\Ca)$ is a trivial Kan fibration. As base change preserves trivial Kan fibrations, we get that $\op{Kart}(\tau)$ is a contractible Kan complex.
    \end{proof}
	We need to give a marked structure on the simplicial sets $\Cart^n$. For this, we need some more notations and maps in the simplicial sets $\Cart^n$.
	
	\begin{notation}
		\begin{enumerate}
			\item We have a map:\[\pi^n : \Crt^n \to [n] \times [n] \] defined as: 
			\[ \pi^n(P):= (\op{min}_{(p,q) \in P}p,\op{min}_{(p,q)\in P}q). \]
			\item $\pi^n \circ \varsigma^n = \op{id}_{[n]\times[n]}$.
			\item Other than $\varsigma^n$, we have two maps: \[ \xi^n,\eta^n: [n] \times [n] \to \Crt^n \] defined by 
			\[\xi^n(p,q):= \varsigma^n(p,0)\wedge \varsigma^n(0,q) ; \eta^n(p,q):= \varsigma^n(p,n)\wedge\varsigma^n(n,q).\]
			\item For $(p,q) \in [n] \times [n]$, we denote \[ \boxplus^n_{(p,q)}:= N(\Crt^n_{\xi^n(p,q)//\eta^n(p,q)}) \]
			\item Denote 
			\[ \boxplus^n:= \cup_{(p,q)\in [n] \times [n]} \boxplus^n_{(p,q)}. \]

		\end{enumerate}
	\end{notation}
	\begin{remark}
		Some remarks on the above notations:
		\begin{enumerate}
			\item  The functors $\xi^n$ and $\eta^n$ satisfy the following property:
			\[ \xi^n(p,q)\le \varsigma^n(p,q) \le \eta^n(p,q) \] 
			\item The definition of $\boxplus^n$ is analog to the definition of $\bx^n$. In the case of $\bx^n$, we have the functors: $p,q: [n] \to \Cpt^n$ defined as $ p(i) =(0,i)$ and $q(i)=(i,n)$. And $ p(i) \le (i,i) \le q(i)$. The functors $p$ and $q$ are analog to the functors $\varsigma^n,\eta^n$ which motivates defining $\boxplus^n_{(p,q)}$ and $\boxplus^n$ in the similar way one defined $\bx^n_i$ and $\bx^n$.
			
		\end{enumerate}
		
	\end{remark}

\begin{lemma}\label{cartninneranodyne}
    The inclusion $\gamma : \boxplus^n_{\op{cart}}:=\bigcup_{0 \le p \le n} \boxplus^n_{(p,n)} \hookrightarrow \Cart^n$ is an inner anodyne.
\end{lemma}
\begin{proof}
We would like to apply \cref{inneranodynepartiallyorderedsetcond} to the elements $p_i = \xi^n(i,n)$ and $q_i = \eta^n(i,n)$ We have $p_{i+1} \le q_i$ as :
\begin{equation}
    \xi^n(i+1,n) \le \sigma^n(i+1,n) \wedge \sigma^n(p,n) = \sigma^n(p,n) = \eta^n(p,n).
\end{equation}
All we need to check $\cup_{p=0}^n \op{Cart}^n_{\xi^n(p,n)//\eta^n(p,n)} = \op{Cart}^n$. Let $Q \in \op{Cart}^n$. Let $p =\pi^n_1(Q)$. Then 
\begin{equation}
    \xi^n(p,n) = \sigma^n(p,0) \wedge \sigma^n(0,n) \le \sigma^n(p,0) \le Q \le \sigma^n(p,n) = \eta^n(p,n).
\end{equation}
Thus the conditions are verified and we get the result.

\end{proof}

So far given an $n$-simplex of $\dd^*_2\Ca_{\E_1,\E_2}$, we consider the simplicial set $\Krt(\tau)$ which are maps $\Cart^n \to \Ca$. But we would like to consier combinatorial simplicial sets mapping to $\dd^*_2\Ca_{\E_1,\E_2}$ and its variants.

	\begin{notation}
\begin{enumerate}
    \item 	Consider the bi-marked simplicial set $(\Delta^n \times \Delta^n, \F'_1:=(\epsilon^2_1 \Delta^{n,n})_1, \F'_2:=(\epsilon^2_2 \Delta^{n,n})_1)$. Let $(\Cart^n, \F)$ be the marked-simplicial set  $(\Cart^n, \F_1:=(\pi^n)^{-1}(\F'_1), \F_2:= (\pi^n)^{-1}(\F'_2))$. \\
		We define $(\Cart^n,\F^{\op{cart}})$ to the $2$-tiled simplicial set where the $2$-tiling is given by $\F_{12}:=\F_1 \star^{\op{cart}} \F_2$.
    \item For any $i \ge -2$, let $(\Ca,\{\E_1,\E_2\},\E_{12}^i)$ be the following $2$-tiled simplicial set where $\E_{12}^i$  is the collection of squares of the form :
    \begin{equation}
        \begin{tikzcd}
            X_0 \arrow[r,"h"] \arrow[d,"w"]&  Y_0 \arrow[d,"f"] \\
            X_1 \arrow[r,"g"] & Y_1
        \end{tikzcd}
    \end{equation}
    where $f,h \in \E_1$ and $g,w \in \E_2$ such taking the pullback of $f$ along $g$ decomposes the square into a diagram of the form
    \begin{equation}
        \begin{tikzcd}
            X_0 \arrow[dr,"d"] \arrow[drr,"h",bend left=40] \arrow[ddr,"w",swap,bend right = 40] & {} & {} \\
            {} & X_0' \arrow[r,"g'"] \arrow[d,"f'"] & Y_0 \arrow[d,"f"] \\
            {} & X_1 \arrow[r,"g"] & Y_1
         \end{tikzcd}
    \end{equation}
    where $d \in \E_1 \cap \E_2$ is $i$-truncated.  Let $\dd^*_2\Ca^i_{\E_1,\E_2}$ be the sub-simplicial set of $\dd^*_2\Ca_{\E_1,\E_2}$ whose $n$ -simplices are $n \times n$ grids where each square of grid admits the above following decomposition.
\end{enumerate}	
	\end{notation}

    \begin{remark}\label{truncationforthmB}
    Some remarks on the above notations :
\begin{enumerate}
    \item By definition, it follows that $\dd^*_2\Ca^{-2}_{\E_1,\E_2} = \dd^*_2\Ca^{\op{cart}}_{\E_1,\E_2}$.
    
    \item  As any square admits a decomposition with an $i$-truncation from the assumption in the conditions of \cref{thmB}, we see that
    \begin{equation}
        \dd^*_2\Ca_{\E_1,\E_2}= \bigcup_{i \ge -2} \dd^*_2\Ca^i_{\E_1,\E_2}.
    \end{equation}
\end{enumerate}
    \end{remark}
	\begin{proposition}\label{truncatedCartalphaconstruction}
Let $\tau : \Delta^n \times \Delta^n \to \Ca$ be a morphism corresponding to a $n$-simplex $\Delta^n \to \dd^*_2\Ca^i_{\E_1,\E_2}$ where $\E_1,\E_2$ is an admissible pair of edges. Let $\tau'$ be an element of $\Krt(\tau)$. Then $\tau'$ induces morphisms $\tau^n_i(\tau')$ and $\tau^{n}_{i-1}(\tau')$ such that the following diagram commutes :
\begin{equation}
    \begin{tikzcd}
        \delta^*_2\delta^{2\bx}_*(\Cart^n,\F^{\op{cart}}) \arrow[d,hookrightarrow] \arrow[r,"\tau^n_{i-1}(\tau')"] & \dd^*_2\Ca^{i-1}_{\E_1,\E_2}\arrow[d,"p^i_{\op{cart}}"]\\
        \dd^*_2\dd^{2+}_*(\Cart^n,\F) \arrow[r,"\tau^n_i(\tau')"] & \dd^*_2\Ca^i_{\E_1,\E_2}.
    \end{tikzcd}
\end{equation}
Morever, we have the following commutative diagram
\begin{equation}
    \begin{tikzcd}
        \Krt(\tau) \arrow[r," \tau^n_{i-1}"] \arrow[dr,"\tau^n_i"] & \op{Fun}( \delta^*_2\delta^{2\bx}_*(\Cart^n,\F^{\op{cart}}),\dd^*_2\Ca^{i-1}_{\E_1,\E_2}) \arrow[d,hookrightarrow]\\
        {} & \op{Fun}(\dd^*_2\dd^{2+}_*(\Cart^n,\F),\dd^*_2\Ca^i_{\E_1,\E_2})
    \end{tikzcd}
\end{equation}
    \end{proposition}

    \begin{remark}
        The above proposition the map $\tau'$ when applying the functors $\dd^*_2$ and $\dd^{2+}_*$ preserves the truncated conditions in the squares and even for particular squares it decreases the truncation.
    \end{remark}

\begin{proof}
The morphism $\tau'$ does induce a morphism :
\begin{equation}
    \tau'' : \dd^*_2\dd^2_*\Cart^n \to \dd^*_2 \dd^{2}_*\Ca
\end{equation}

We need to show that $\tau''$ induces $\tau^n_i(\tau')$ and $\tau^n_{i-1}(\tau')$ respectively. In particular, this reduces to check $\tau''$ maps the corresponding markings of edges and tilings as desired. 
\begin{itemize}
    \item \textbf{$\tau''$ sends $\F$ to $(\E_1,\E_2)$}: Let $e :x \to y$ be an edge in $\op{Cart}^n$ lying in $\F_1$ this means $\pi^n_1(x)=\pi^n_1(y)=p$. Using \cref{morphisminUPinexactsquare}, we see that $e$ is a finite compositions of exact pullback of morphisms given by :
    \begin{enumerate}
        \item $\sigma^n(n,q) \to \sigma^n(n,q)-(n,q) =\sigma^n(n,q-1)$
        \item and also of the form $\sigma^n(p,q) \to \sigma^n(p,q)-\{(p,q)\}$ where $1 \le p,q \le n-1$.
    \end{enumerate}
   1.  The first one after applying $\tau''$ gives us as a morphism in $\E_1$ as $\tau$ preserves edges. \\
   2.  The second one fits into the following diagram 
    \begin{equation}\label{truncatedcartdiagram}
        \begin{tikzcd}
            \sigma^n(p,q) \arrow[dr,"e'"] & {} & {} \\
            {} & \sigma^n(p,q)-\{(p,q)\} \arrow[r,"e_1"] \arrow[d] & \sigma^n(p,q+1) \arrow[d]\\
            {} & \sigma^n(p+1,q) \arrow[r] & \sigma^n(p+1,q+1).
        \end{tikzcd}
    \end{equation}
    where the square is a pullback square. As $\tau'$ is a RKE of $\tau$, $\tau'$ shall send the exact square to the pullback square (\cref{KrtcontKancomplex}).  Thus in $\Ca$ after aplpying $\tau'$ we see that $\tau''$ sends $e'$ to an element in $\E_1$ as edges in $\E_1$ are admissible and  $\tau'(e_1 \circ e)$ and $\tau'(e_1)$ are in $\E_1$.   \\
    
    The similar argument follows for $\F_2$.
    
    \item \textbf{$\tau''$ sends tiling by $\F$ to tiling by $(\E_1,\E_2)^i$}: Let 
    \begin{equation}
    \begin{tikzcd}
        x_1 \arrow[r] \arrow[d] & y_1 \arrow[d]\\
        x_2 \arrow[r] & y_2
    \end{tikzcd}
    \end{equation}
    be a square in $\Cart^n$ where the horizontal edges are in $\F_1$ and vertical edges are in $\F_2$ and none of the edges are degenerate. Then it follows that $\pi(x_1)=(r_1,r_2)$ where $r_1,r_2 \le n-1$. Let $x_1'$ be the fiber product of the square and we are interested in showing that $\tau'$ applied to  $x_1 \to x_1'$ is $i$-truncated. By \cref{morphisminUPinexactsquare}, we see that $x \to x'$ is composed of exact pullbacks by $\sigma^n(p,q) \to \sigma^n(p,q)-\{(p,q)\}$. \\
    By \cref{truncatedcartdiagram}, we see that $\tau'(\sigma^n(p,q) \to \sigma^n(p,q)-\{(p,q)\})$ is same as $\tau'(\sigma^n(p,q)) \to \tau'(\sigma^n(p,q+1)) \times_{\tau'(\sigma^n(p+1,q+1)} \tau'(\sigma^n(p+1,q))$. By definition of $\dd^*_2\Ca^i_{\E_1,\E_2}$, we see that such an edge is $i$-truncated. As $i$-truncated morphisms are stable under and compositions, we see that $\tau'(x_1 \to x_1')$ is $i$-truncated.
    \item \textbf{$\tau''$ sends $\F^{\op{cart}}$ to $(\E_1,\E_2)^{i-1}$:}

    Let  \begin{equation}
    \begin{tikzcd}
        x \arrow[r] \arrow[d] & y \arrow[d]\\
        z \arrow[r] & w
    \end{tikzcd}
    \end{equation}
    be a square in $\Cart^n$ where the horizontal edges are in $\F_1$ and vertical edges are in $\F_2$ and none of the edges are degenerate. Also further assume that the square is a pullback square. We assume $x_1 = y\wedge z$. We decompose the square in the following diagram : 
    \begin{equation}
        \begin{tikzcd}
            y \wedge z  \arrow[r] \arrow[d] & z \arrow[d] & {} \\
            y \arrow[r] & y \vee z \arrow[dr] & {}\\
            {} & {} & w
        \end{tikzcd}
    \end{equation}

    By \cref{truncationdecreased} and the fact that $\tau'(y \cap z \to w)$ is $i$-truncated from the previous point, we get that the canonical map
    \begin{equation}
        \tau'(x= y\vee z \to y \times_w z)
    \end{equation}
    is $i-1$-truncated. This show that $\tau''$ sends cartesian squares to squares with one less truncation.
\end{itemize}

\end{proof}

Following the ideas in extension along $p_{\op{comm}}$, we would like to have the map 
\begin{equation}
    j : \dd^*_2\dd^{2\bx}_*(\Cart^n,\F^{\op{cart}}) \hookrightarrow \dd^*_2\dd^{2+}_*(\Cart^n,\F)
    \end{equation}
    to be inner anodyne. Instead of showing this,  our aim is to show the existence of the following commutative diagram :
    \begin{equation}\label{inneranodynecommutativecartsquare}
    \begin{tikzcd}
        \boxplus^n_{\op{cart}}\arrow[r,"\epsilon^{\op{cart}}"] \arrow[d,hookrightarrow] \arrow[d,"\gamma"] & \dd^*_2\dd^{2\bx_*}(\Cart^n,\F^{\op{cart}}) \arrow[d,hookrightarrow,"j"] \\
        \Cart^n \arrow[r,"\epsilon"] & \dd^2_*\dd^{2+}_*(\Cart^n,\F)
    \end{tikzcd}
     \end{equation}
          where $\gamma$ is proved to be inner anodyne (\cref{cartninneranodyne}).  This motivates to define the maps $\epsilon$ and $\epsilon^{\op{cart}}$
     \begin{construction}
    For two elements $x,y \in \op{Cart}^n$ and for $p,q \in [n]$, we define two elements:
         \begin{equation}
             \Lambda_p^n(x,y) := (\sigma^n(\pi^n_1(y) \vee p,0) \vee x) \wedge y \quad ; \quad \mu_q^n(x,y) := (\sigma^n(0,q \vee \pi^n_2(y)) \vee x) \wedge y.
         \end{equation}
          \end{construction}
\begin{lemma}\label{lambdamuprop}
    Considering the notation above, we have 
    \begin{enumerate}
        \item $\Lambda_p^n(x,x)=x = \mu_q^n(x,x)$.
        \item $\pi^n(\Lambda^n_p(x,y))= (\pi^n_1(y),\pi^n_2(x))$ amd $\pi^n(\mu^n_q(x,y)) = (\pi^n_1(x),\pi^n_2(y))$.
    \end{enumerate}
\end{lemma}
\begin{proof}
    \begin{enumerate}
        \item $\Lambda^n_p(x,x) = \sigma^n(\pi^n_1(x) \vee p,0) \vee x \wedge x = \sigma^n(\pi^n_1(x) \vee p,0) \vee x = x$. The same argument holds for $\mu^n$.
        \item  
    \end{enumerate}
\end{proof}
\begin{construction}\label{epsilonconstruction}
    We define 
    \begin{equation}
        \epsilon_n : \Cart^n \to \dd^*_2 \dd^{2+}_*(\Cart^n,\F)
    \end{equation}
    as follows: \\
    Let $ r : \Delta^m \to \op{Cart}^n$ which maps the vertices to $x_i$ for all $ 0 \le i \le r$ Then 
    \begin{equation}
        \epsilon(r) : \Delta^m \times \Delta^m \to \Cart^n \quad (a,b) \mapsto \begin{cases}
             \Lambda^n_0(x_b,x_a) \quad a \ge b \\
             \mu^n_0(x_a,x_b) \quad a \le b \\
        \end{cases}
    \end{equation}
By \cref{lambdamuprop}, we see that the map $\epsilon_n$ maps edges $\F_1,\F_2$ to $\E_1$ and $\E_2$ respectively.
    \end{construction}

\begin{remark}
    In particular, for $m=1$ and for any edge $x \to y \in \Cart^n$, $\epsilon_n$ maps such an edge to the following square :
    \begin{equation}
        \begin{tikzcd}
            x \arrow[r] \arrow[d] & \mu^n(x,y) \arrow[d] \\
            \Lambda^n(x,y) \arrow[r] & y.
        \end{tikzcd}
    \end{equation}
\end{remark}

\begin{notation}
Recall that $\boxplus^n_{\op{cart}}$ be the simplicial set: 
\begin{equation}
    \bigcup_{0 \le p \le n}\boxplus^n_{(p,n)}
\end{equation}
   
\end{notation}
\begin{claim}
    The map $\epsilon_n$ induces :
    \begin{equation}
        \epsilon^{\op{cart}}_n : \boxplus^n_{\op{cart}} \to \dd^*_2\dd^{2\bx}_*(\Cart^n, \F^{\op{cart}}).
    \end{equation}
\end{claim}
\begin{proof}
Let $\tau$ be an $n$-simplex of $\boxplus^n_{\op{cart}}$. We want to show that for any map $\tau':\Delta^1 \times \Delta^1 \to \Delta^m$, the map $\epsilon^{\op{cart}}_n $sends this square to a pullback square. Let the vertices of $\tau'$ given by $(a,b), (a+1,b),(a,b+1),(a+1,b+1)$.

\begin{enumerate}
    \item $\textbf{a=b:}$ Let $\Lambda^n_0(x_a,x_a)=x$ and $\Lambda^n$ and $\Lambda^n_0(x_{a+1},x_{a+1})=y$ (\cref{lambdamuprop}),we want to show that :
    \[\Lambda^n_0(x,y) \times_{y} \mu^n_0(x,y) =x. \]
    By \cref{cartdiagrampullbacksquare}, we see that $\Lambda^n_a(x,y) \times_{y} \mu^n_0(x,y) \to \Lambda^n_0(x,y) \wedge \mu^n_0(x,y)$ is an isomorphism as $ y \to y \wedge y = y$ is an isomorphism. Thus we are reduced to show :
    \begin{equation}
        \Lambda^n_0(x,y) \wedge \mu^n_0(x,y) = x
    \end{equation}
     We know that $\xi^n(p,n) \le x_a \le x_b \le \eta^n(p,n)$. This means $\pi^n(y) \le (p,n)$.

     This gives that
    \begin{align*}
        \Lambda^n_0(x,y) \wedge \mu^n_0(x,y)
        \\ = (\sigma^n(\pi^n_1(y),0) \vee x \wedge y) \wedge (\sigma^n(0,\pi^n_2(y))\vee x \wedge y) 
        \\= (\sigma^n(\pi^n_1(y),0) \wedge (\sigma^n(0,\pi^n_2(y))) \vee x \wedge y 
        \\ \le x \vee x \wedge y = x
    \end{align*}

    As by property of pullback we already had $\Lambda^n_0(x,y) \wedge \mu^n_0(x,y) \ge x$. This proves that the square is pullback. \\

    \item \textbf{$a < b$ :} We need to show that the square :
    \begin{equation}
        \begin{tikzcd}
            \Lambda^n_0(x_a,x_b) \arrow[r] \arrow[d] & \Lambda^n_0(x_a,x_{b+1}) \arrow[d]\\
            \Lambda^n_{0}(x_{a+1},x_b) \arrow[r] & \Lambda^n_{0}(x_{a+1},x_{b+1})
        \end{tikzcd}
    \end{equation}
    is a pullback square. Let $m = \pi^n_1(x_b)$ and $m' = \pi^n_1(x_{b+1})$. The following pullback condition follows from the following diagram :
    \begin{equation}
        \begin{tikzcd}
            \sigma^n(m,0) \vee x_a \wedge x_b \arrow[r] \arrow[d] & \sigma^n(m,0) \vee x_{a} \wedge x_{b+1} \arrow[r] \arrow[d] & \sigma^n(m',0) \vee x_a \wedge x_{b+1} \arrow[d] \\
            \sigma^n(m,0) \vee  x_{a+1} \wedge x_b \arrow[r] \arrow[d] & \sigma^n(m,0) \vee x_{a+1} \wedge x_{b+1} \arrow[r] \arrow[d] & \sigma^n(m',0) \vee x_{a+1} \wedge x_{b+1} \arrow[d] \\
            \sigma^n(m,0) \vee x_{a+1} \wedge x_b \arrow[r] & \sigma^n(m,0) \vee x_{a+1} \wedge x_{b+1} \arrow[r] & *
        \end{tikzcd}
    \end{equation}
    where each of the smaller squares are pullback squares hence the outer square is which proves the desired claim.
    \item  For the other case, it same argument with $\mu^n(x,y)$.
    \end{enumerate}

\end{proof}

\begin{remark}
    Let us analyze \cref{truncatedCartalphaconstruction} and \cref{inneranodynecommutativecartsquare} combined for $n=1$. We have the following diagram corresponding to a one simplex $\tau$ of $\dd^*_2\Ca^i_{\E_1,\E_2}$ of the form :
    \begin{equation}\label{exampleofcart1}
            \begin{tikzcd}
        x \arrow[r] \arrow[d] & y \arrow[d]\\
        z \arrow[r] & w.
    \end{tikzcd}
    \end{equation}
    as follows :
    \begin{equation}
          \begin{tikzcd}
        {} & \boxplus^1_{\op{cart}}\arrow[r,"\epsilon^{\op{cart}}"] \arrow[d,hookrightarrow] \arrow[d,"\gamma"] & \dd^*_2\dd^{2\bx_*}(\Cart^n,\F^{\op{cart}}) \arrow[d,hookrightarrow,"j"] \arrow[r,"\tau^n_{i-1}(\tau')"] & \dd^*_2\Ca^{i-1}_{\E_1,\E_2} \arrow[d,"p^{i-1,i}_{\op{cart}}"] \\
       \Delta^1 \arrow[r] &\Cart^1 \arrow[r,"\epsilon"] & \dd^2_*\dd^{2+}_*(\Cart^n,\F) \arrow[r,"\tau^n_i(\tau')"] & \dd^*_2\Ca^i_{\E_1,\E_2}
    \end{tikzcd}
    \end{equation}
    Here $\tau'$ is of the form :
    \begin{equation}
    \begin{tikzcd}
        x \arrow[dr] & {} & {}\\
        {} & x' \arrow[d] \arrow[r] & y \arrow[d] \\
        {} &  z \arrow[r] & w
    \end{tikzcd}
\end{equation}
where $x \to x'$ is in $\E_1 \cap \E_2$ and $x \to x'$ is $i$-truncated. 
\begin{itemize}
    \item The bottom arrow corresponds to the square \cref{exampleofcart1}.
    \item In the top row, note that $\boxplus^1_{\op{cart}}$ consists of two simplicial sets $\boxplus^1_{0,1}$ amd $\boxplus^1_{1,1}$. 
    \begin{enumerate}
        \item $\boxplus^1_{0,1}$ consists of edge of the form $\sigma^1(0,0) \to \sigma^(0,1) \wedge \sigma^(1,0)$. Traversing through the top arrow, this edge maps to the square :
        \begin{equation}
            \begin{tikzcd}
                x \arrow[r,"\op{id}"] \arrow[d,"\op{id}"] & x \arrow[d] \\
                x \arrow[r] & x'
            \end{tikzcd}
        \end{equation}
        Note that this makes sense as $x \to x \times_x' x$ is $i-1$-truncated. 
        \item $\boxplus^1_{(1,1)}$ consists an edge of the form $\sigma^1(0,1) \wedge \sigma^(1,0) \to \sigma^1(1,1)$. The top arrow sends this edge to the square :
        \begin{equation}
                \begin{tikzcd}
        x' \arrow[r] \arrow[d] & y \arrow[d]\\
        z \arrow[r] & w.
    \end{tikzcd}
        \end{equation}
    \end{enumerate}
    \item On composing with $g^i_{\op{cart}}$, we see tha above two edges corresponds to $\Lambda^2_1 \to \D$. Using the fact that $\gamma$ is an inner anodyne (\cref{cartninneranodyne}), we see that this allows us to map the $1$-simplex $\tau$ in $\D$. Thus the above simplicial sets do recover the same idea explained in the beginning of the section. 
\end{itemize}
    
\end{remark} 

\subsection{Proof of Theorem B.}
We restate the theorem from the introduction :

\begin{theorem}[Theorem B : Extension along $p_{\op{cart}}$]\label{thmB}
    		Let $\Ca$ be an $\infty$-category and $\E_1,\E_2$ be a collection of edges in $\Ca$ with the following conditions:
  \begin{enumerate}
\item Every morphism $f \in \E_1\cap \E_2$ is $k$-truncated for $k \ge -2$.

\item The edges $\E_1$ and $\E_2$ are admissible.

  \end{enumerate}
  Then for any $\infty$-category $\D$, there exists a solution to the lifting problem:
\begin{equation}
\begin{tikzcd}
    \dd^*_2\Ca^{\op{cart}}_{\E_1,\E_2} \arrow[d,"p_{\op{cart}}",swap] \arrow[r,"g_{\op{cart}}"] &\D\\
    \dd^*_2\Ca_{\E_1,\E_2} \arrow[ur,"g'_{\op{cart}}",swap,dotted] & {}.
    \end{tikzcd}
\end{equation}
    
  \end{theorem}
\begin{proof}
 Similar to the proof of \cref{thmA}, we verify the conditions of \cref{maintechnicalsimplicesthm} to get desired extension $g_{\op{cart}}'$. Before doing that, we need to reduce the problem to setting of specific truncation of morphisms. 

 Let $p^i_{\op{cart}}$ be the canoncial inclusion $\dd^*_2\Ca^{-2}_{\E_1,\E_2} \to \dd^*_2\Ca^i_{\E_1,\E_2}$. By \cref{truncationforthmB}, it follows that we are reduced to construct the existence of the dotted arrow for the following diagram :
 \begin{equation}
     \begin{tikzcd}
         \dd^*_2\Ca^{-2}_{\E_1,\E_2} \arrow[r,"g^{-2}_{\op{cart}}"] \arrow[d,"p^{i}_{\op{cart}}",swap] & \D \\
         \dd^*_2\Ca^i_{\E_1,\E_2} \arrow[ur,dotted, "g^i_{\op{cart}}",swap] & {}
     \end{tikzcd}
 \end{equation}
 Solving this extension problem defines $g'_{\op{cart}}:= \bigcup_{i\ge -2} g^i_{\op{cart}}$. We proceed by induction on $i$. Thus by induction we consider, that the above diagram decomposes into the following diagram:
 \begin{equation}
     \begin{tikzcd}
         \dd^*_2\Ca^{-2}_{\E_1,\E_2} \arrow[r, "g^{-2}_{\op{cart}}"] \arrow[d,"p^{i-1,2}_{\op{cart}}",swap] & \D \\
         \dd^*_2\Ca^{i-1}_{\E_1,\E_2} \arrow[ur,"g^{i-1}_{\op{cart}}",swap] \arrow[d,"p^{i-1,i}_{\op{cart}}",swap] & {} \\
         \dd^*_2\Ca^i_{\E_1,\E_2} \arrow[uur,bend right =80 ,"g^i_{\op{cart}}",dotted] & {}
     \end{tikzcd}
 \end{equation}
 We shall now verify the conditions of \cref{maintechnicalsimplicesthm} now for the above lifting problem.

 \begin{itemize}
     \item  Let $\tau$ be an $n$-simplex of $\dd^*_2\Ca^i_{\E_1,\E_2}$ which is a simplex of the form $\tau : \Delta^n \times \Delta^n \to \Ca$. Define $\alpha'_n$ as the following chain of compositions :
     \begin{equation}
         \alpha'_n : \op{Kart}(\tau) \xrightarrow{\tau^n_{i-1}} \op{Fun}( \delta^*_2\delta^{2\bx}_*(\Cart^n,\F^{\op{cart}}),X) \xrightarrow{h} \op{Fun}(\boxplus^n_{\op{cart}},\D) 
     \end{equation}
    where : 
    \begin{enumerate}
        \item $X =,\dd^*_2\Ca^{i-1}_{\E_1,\E_2}$.
        \item $h :\op{Fun}(\delta^*_2\delta^{2\bx}_*(\Cart^n,\F^{\op{cart}}),X) \xrightarrow{\epsilon^{\op{cart}}_n} \op{Fun}(\boxplus^n_{\op{cart}},X) \xrightarrow{g^{i-1}_{\op{cart}}} \op{Fun}(\boxplus^n_{\op{cart}},\D)$
         \end{enumerate}
Define $\N(\tau) : = \Krt(\tau) \times_{\op{Fun})\boxplus^n_{\op{cart}},\D)} \op{Fun}(\Cart^n,\D)$. Then let $\alpha_n$ be the morphism :
\begin{equation}
    \alpha_n : \N(\tau) \to \op{Fun}(\Cart^n,\D) \xrightarrow{i \circ \sigma^n} \op{Fun}(\Delta^n,\D)
\end{equation}
where $i : \Delta^n \to \Delta^n \times \Delta^n$ is the diagonal map. By \cref{cartninneranodyne}, we see that the restriction map $\op{Fun}(\Cart^n,\D) \to \op{Fun}(\boxplus^n_{\op{cart}},\D)$ is a trivial Kan vibration. Hence the base change map $\N(\tau) \to \Krt(\tau)$ is a trivial Kan vibration. This implies that $\N(\tau)$ is a contractible Kan complex hence it is weakly contractible. 
         
    \item As $\Krt(\tau)$ is a contractible Kan complex(\cref{KrtcontKancomplex}), we see that $\alpha_n$ can be realized as a map :
    \begin{equation}
        \alpha_n : \N(\tau) \to \op{Fun}^{\simeq}(\Delta^n,\D)
    \end{equation}
    Functoriality of these combinatorial simplicial sets gives us the map :
    \begin{equation}
        \alpha : \N \to \op{Map}[\dd^*_2\Ca^i_{\E_1,\E_2},\D]
    \end{equation}
    \item  Suppose $\tau$ comes from an $n$-simplex $\tau'$ of $\dd^*_2\Ca^{-2}_{\E_1,\E_2}$. Let $\omega^1_n = \tau'$ which is an element of $\Krt(\tau)$ as pullback squares Kan extended yield the same morphisms. The element $\omega_n$ along with the \cref{truncatedCartalphaconstruction} and \cref{inneranodynecommutativecartsquare}, we have the following chain of morphisms :

    \begin{equation}
       \alpha^2_n: Krt(\tau) \to \op{Fun}(\dd^*_2\dd^{2+}_*(\Cart^n,\F),\dd^*_2\Ca^{-2}_{\E_1,\E_2}) \xrightarrow{g^{-2}_{\op{cart}} \circ \epsilon_n} \op{Fun}(\Cart^n,\D)
    \end{equation}
One checks following the construction of maps that the pair $(\omega^1_n,\alpha^2_n(\omega_1)) \in \N(\tau)$ gets mapped to $g^{-2}_{\op{cart}}(\tau')$ via $\alpha_n$. \\
The collection $\omega_n$ defines an element $\omega \in \Gamma(p^{i*}_{\op{cart}}\N)_0$ such that $\Gamma(p^{i*}_{\op{cart}}\alpha)(\omega) = g_{\op{cart}}$. This proves the compatibility with $p^i_{\op{cart}}$.
    \item As we verified the conditions of \cref{maintechnicalsimplicesthm}, we get the existence of dotted arrow $g^i_{\op{cart}}$ such that the diagram commutes.
 \end{itemize} 
 \end{proof}

 \begin{remark}
The proof of \cref{thmA} and \cref{thmB} completes the proof of \cref{compthm}. 
  \end{remark}

%% file: appendices.tex
\begin{appendices}

\section{Inner anodyne maps between partially ordered sets.}\label{partiallyorderedsetsdefinitionandprop}

\begin{definition}
    Let $P$ be a partially ordered set. A \textit{lattice} is a partially ordered set which admits products (infima) and coproducts (suprema) for a finite number of elements. We shall denote coproducts by $\wedge$ and products by $\vee$. A lattice is said to be \textit{distributive} if for three elements $p,q,r \in P$, we have $p \vee (q \wedge r) = (p \vee q) \wedge (p \vee r)$. A \textit{sublattice} of a lattice $P$ is a subset $Q \subseteq P$ which is stable under finite coproducts and products.
\end{definition}

\begin{example}
    For a lattice $P$ and $p,q \in P$, the undercategory $P_{p/}$, the overcategory $ P_{/p}$ and the partially ordered set $P_{p//q} = \{ x \in P~|~p \le x \le q \}$ are sublattices.
 \end{example}
 \begin{definition}
     Let $P$ be a partially ordered set. A subset $Q \subseteq P$ is said to be a \textit{up-set} if for $p,q \in P$ where $p \in Q$, then $p \le q \implies q \in Q$.
 \end{definition}
 \begin{example}
     The  undercategory $P_{p/}$ is a typical example of an up-set of $P$.
 \end{example}
 \begin{notation}
      Let $P \subset Q, P \subset R$ be two full inclusions of partially ordered sets.  Let $S:= Q \coprod_P R$ be the set theoretic pushout.  The set $S$ is a partially ordered set with the following properties:
      \begin{enumerate}
          \item The subsets $Q$, $R$ are full inclusions of partially ordered sets in $S$. Let $i^Q_S$ and $i^R_S$ be the respective inclusions.
          \item Let $q \in Q$ and $r \in R$, then $q \le r$ iff $ \exists~ p \in P$ such that $ q \le p \le r$.
      \end{enumerate}
 \end{notation}
 One has the inclusion 
 
 \begin{equation}
     i^{Q,R}_S: N(Q)\cup N(R) \hookrightarrow N(S).
 \end{equation}The following lemma provides some conditions in which the following map $i^{Q,R}_S$ is an inner anodyne.
 \begin{proposition}\label{inneranodynepartiallyorderedsetcond}
     Let $P,Q,R$ and $S$ be partially ordered sets defined in the notation above. Suppose we have:
     \begin{enumerate}
         \item $Q$ admits pushouts and pushouts are preserved by the inculsion $i_Q$.
         \item $Q-P$ is finite.
         \item $P$ is an up-set of $Q$.
     \end{enumerate}
     Then the inclusion 
     \begin{equation}
         i^{Q,R}_S : N(Q)\cup N(R) \to N(S)
     \end{equation}
     is an inner anodyne. 
 \end{proposition}
 \begin{proof}
We prove this by induction on $Q-P$. 

\begin{enumerate}
    \item \textbf{$|Q-P|=1:$} 
    \begin{itemize}
    \item Let $q \in Q-P$, as $P$ is an upset, it turns out that $q$ is a minimal element of $Q$ and by the partially ordered on $S$, we have $q$ as a minimal element of $S$. Considering the over categories $P_{q/},Q_{q/},R_{q/},S_{q/}$. As $q$ is a minimal element, it turns out that $Q_{q/} = P_{q/}^{\triangleleft}$ and $S_{q/}= R_{q/}^{\triangleleft}$.  By the definition of again $q$ as a minimal element, we have the following  diagram :
    \begin{equation}
    \begin{tikzcd}
        N(P_{q/})^{\triangleleft}\coprod_{N(P_{q/})} N(R_{q/}) \arrow[r] \arrow[d,"i^{Q,R}_{S,q}"] & N(Q) \cup N(R) \arrow[d,"i^{Q,R}_S"] \\
        N(R_{q/})^{\triangleleft} \arrow[r] & N(S).
        \end{tikzcd}
    \end{equation} 
   \item  We claim that the square is a pushout square.
    In order to prove the following pushout, the pushout of the square is $N(R) \cup N(S_{q/})$. An $m$-simplex of $N(S)$ can either containg $q$ as a vertex or not. If it does not contain $q$, then the simplex is in $N(R)$. If it contains $q$ as a vertex, then the $0$th vertex of $m$-simplex is $q$ (as $q$ is minimal). Such a simplex lies in $N(S_{q/})$. Thus we have $N(R) \cup N(S_{q/}) =N(S)$. This proves that the above square is pushout.\\

  \item  As inner anodyne maps are preserved under pushouts, we show that $i^{Q,R}_{S,q}$ is an inner anodyne. Using \cite[Lemma 2.1.2.3]{HTT}, it is enough to show that the morphism $N(P_{q/}) \to N(R_{q/})$ is left anodyne. 

  \item We use the theory of cofinal maps to show that $N(P_{q/})^{op} \to N(R_{q/})^{op}$ is cofinal map (hence a right anodyne). By \cite[Theorem 4.1.3.1]{HTT}, this is equivalent to show that for all $r \in R_{q/}$, the category $(P_{q//r})^{op} = (P_{q/} \times_{R_{q/}} r)^{op}$ is weakly contractible. Condition $(1)$ of the proposition gives that the category $P_{q//r}$ admits coproducts which implies that its opposite category admits products which implies it is weakly contractible. This completes the proof of $|Q-P|=1$.
 \end{itemize}
    \item \textbf{$|Q-P|=n-1 \implies |Q-P|=n$:} Let $q \in Q-P $ be a minimal element. Let $Q'=Q-\{q\}$ $S'=S-\{q\}$, we have the following diagram : 
    \begin{equation}
    \begin{tikzcd}
      N(Q') \cup N(R) \arrow[d,"i^{Q',R}_{S'}"] \arrow[r]  & N(Q) \cup N(R) \arrow[d,"i_1"] \arrow[dr,"i^{Q,R}_S"] & {} \\
      N(S') \arrow[r] & N(Q) \cup N(S') \arrow[r,"i^{Q,S'}_S",swap] & N(S)
      \end{tikzcd}
    \end{equation}
    \begin{itemize}
        \item  As $S':= Q'\coprod_P R$, using the induction on $|Q'-P|=n-1$ and $Q'$ admits pushouts and stable under the inclusion $Q' \hookrightarrow S'$ and $P$ is still an upset of $Q'$, we see that by induction $i^{Q',R}_{S'}$ is inner anodyne. As the square is pushout, we see that $i_1$ is an inner anodyne.
        \item  Notice that $S= Q \coprod_{Q \cap S'} S'$ and $|Q-{Q \cap S'}|= 1$. Similar to the other point, we see that the other conditions  of proposition hold in order to apply the induction for $n=1$. Thus we see that $i^{Q,S'}_S$ is an inner anodyne.
        \item Composing the two maps, we see that $i^{Q,R}_S$ is an inner anodyne completing the induction step.
    \end{itemize}
    
\end{enumerate}
 \end{proof}

 We shall apply the above proposition to prove the following statement which helps to prove properties about the combinatorial simplicial sets defined in the setting of $\infty$-categorical compactification.
 \begin{proposition}\label{partiallyorderedmainprop}
     Let $P$ be a finite partially ordered set. Let $p_1\le p_2 \le \cdots p_l$ and $q_1 \le q_2 \le q_3 \cdots q_l$ such that $p_{j} \le q_{j-1}$ for all $2 \le j \le l$. Then the inclusion:
     \begin{equation}
         i_l: \cup_{j=i}^l N(P_{p_j//q_j}) \hookrightarrow N(\cup_{j=1}^l P_{p_j//q_j})      \end{equation}
         is an inner anodyne.
 \end{proposition}

 \begin{proof}
     The above inclusion $i_l$ is an inclusion of the following maps :

     \begin{equation}
          \cup_{j=i}^l N(P_{p_j//q_j}) \xrightarrow{i_l^2} N(\cup_{j=1}^2 P_{p_j//q_j})\bigcup  \cup_{j=3}^l N(P_{p_j//q_j}) \xrightarrow{i_i^3} \cdots  N(\cup_{j=1}^l P_{p_j//q_j})       \end{equation}
 \end{proof}
 Thus, it is enough to show the map
 \begin{equation}
     i_k : N(\cup_{j=1}^{k-1} P_{p_j//q_j})\cup N(P_{p_k//q_k}) \rightarrow N(\cup_{j=1}^{k} P_{p_j//q_j})
 \end{equation}
 is an inner anodyne.  Let $Q:= \cup_{j=1}^{k-1} P_{p_j//q_j}$, $R:=P_{p_k//q_k}$ and $S:= \cup_{j=1}^{k} P_{p_j//q_j}$. We want to apply \cref{inneranodynepartiallyorderedsetcond} to prove our claim. 
 \begin{itemize}
     \item Let $P':= Q \cap R = P_{p_k//q_{k-1}}$. $Q-P'$ is finite.
     \item By definition of $P'$, we see that $P'$ is an upset of $Q$.
     \item We show that $S$ is the pushout $Q \coprod_{P'} R$. On the level of sets, it is the pushout. We need to check the partial ordering of the set $S$. Clearly the inclusions $P' \subset S$ and $ R \subset S$ preserves orderings. If $x \in Q, y \in R$, then $x \wedge p_k \in P_{p_k//q_{k-1}}=P'$. This implies $x \le x \wedge p_k \le y$. If $x \ge y$, we see that this is only possible if $x,y \in P'$. In these case $x = x \ge y$. This proves $S$ as the pushout in the category of partially ordered sets.
     \item For $x \in P_{p_{k_1}//q_{k_1}}$ and $y \in P_{p_{k_2}//q_{k_2}}$, we see that $x \wedge y \in P_{p_{k_3}//q_{k_3}}$ where $k_3 =\op{max}(k_1,k_2)$.
  \end{itemize}
 \section{Existence of finite limits in overcategories.}

\begin{lemma}\label{overcategorylimitpreserving}
    Let $\Ca$ be an $\infty$-category and $c \in \Ca$ be a point in $\Ca$. Let $B$ be a weakly contractible simplicial set. Then a morphism $p : B \to \Ca_{/c}$ admits a limit iff $p': B \to \Ca_{/c} \to \Ca$ admits a limit. 
\end{lemma}

 \begin{proof}
 We use the fact that for any simplicial set $A$, the inclusion map $B \to B \boldsymbol{*}A$ is left anodyne. This is essentially \cite[Lemma 4.2.3.6]{HTT}.
     \begin{enumerate}
         \item Suppose the map $p : B \to \Ca_{/c}$ admits a limit. Let $\tilde{p} : B^{\triangleleft} \to \Ca_{/c}$ be the corresponding limit diagram.  We shall show that the composition :
         \begin{equation}
             \tilde{p'}: B^{\triangleleft} \to \Ca_{/c} \to \Ca
         \end{equation}
         is a limit diagram of $p'$.  We need to show for all $n \ge 0$, there exists a solution of the lifting problem:
         \begin{equation}
             \begin{tikzcd}
                 \partial\Delta^n \boldsymbol{*} B \arrow[r,"h_n"] \arrow[d, hookrightarrow] & \Ca \\
                 \Delta^n \boldsymbol{*} B \arrow[ur,dotted,"h_n'",swap] & {}
             \end{tikzcd}
             \end{equation}
             where $h_n|_{[n]}= p''$. As $h_n|_{[n]} : B^{\triangleleft} \to \Ca$ lifts to $B^{\triangleleft \triangleright} \to \Ca$ (via the point $c$).
        Using the fact that $B \to B^{\triangleleft}$ is inner anodyne and applying it in \cite[Lemma 2.1.2.3]{HTT}, we have a solution to the lifting problem :
     \begin{equation}
         \begin{tikzcd}
             \partial\Delta^n \boldsymbol{*} B \coprod_{[n]\boldsymbol{*}B} [n] \boldsymbol{*}B^{\triangleright} \arrow[r] \arrow[d,hookrightarrow] & \Ca \\
             \partial\Delta^n \boldsymbol{*} B^{\triangleright} \arrow[ur,dotted] & {}
         \end{tikzcd}
     \end{equation}
     Note that this allows us to see $h_n$ as a morphism $\partial\Delta^n\boldsymbol{*} B \to \Ca_{/c}$. Then by property of $\tilde{p}$ being a limit diagram.  $h_n$ extends to $h_n' : \Delta^n\boldsymbol{*}B \to \Ca_{/c} \to \Ca$. This completes the proof.

     \item Suppose that $p'$ admits a limit diagram $\tilde{p'} : B^{\triangleleft} \to \Ca$. First, we lift $\tilde{p'}$ factorizes via $\Ca_{/c}$. This follows from the solution of the lifting problem using the same arguments in previous point:
     \begin{equation}
         \begin{tikzcd}
             \Delta^0 \boldsymbol{*} B \coprod_{ B} B^{\triangleright} \arrow[d,hookrightarrow] \arrow[r] & \Ca \\
             \Delta^0 \boldsymbol{*} B^{\triangleright} \arrow[ur,dotted] & {}
         \end{tikzcd}
     \end{equation}
     This shows that $\tilde{p'}$ factorizes through $\tilde{p}: B^{\triangleleft}=\Delta^0 \boldsymbol{*}B \to \Ca_{/c}$. We show $\tilde{p}$ is a limit diagram. For $n \ge 1$, we need to show the solution of the lifting problem :
     \begin{equation}
     \begin{tikzcd}
         \partial\Delta^n \boldsymbol{*} B \arrow[r,"g_n"] \arrow[d,hookrightarrow] & \Ca_{/c} \\
         \Delta^n \boldsymbol{*} B \arrow[ur,dotted,"g_n'",swap] & {}
         \end{tikzcd}
     \end{equation}
     where $g_n|_{[n]}= \tilde{p}$.
      Rewriting the solution in terms of category $\Ca$, we need to show the solution of the lifting problem :
      \begin{equation}
      \begin{tikzcd}
          \partial\Delta^n \boldsymbol{*} B^{\triangleright} \arrow[r] \arrow[d,hookrightarrow] & \Ca \\
          \Delta^n \boldsymbol{*} B^{\triangleright} \arrow[ur,dotted] & {}
          \end{tikzcd}
      \end{equation}
      The composition $\partial\Delta^n\boldsymbol{*} B \to \Ca_{/c} \to \Ca$ admits an extension to $\Delta^n \boldsymbol{*} B \to \Ca$ using the fact that $\tilde{p'}$ is a limit diagram. Thus we are reduced to solve the following lifting problem :
      \begin{equation}
      \begin{tikzcd}
          \partial\Delta^n \boldsymbol{*} B^{\triangleright} \coprod_{\partial\Delta^n \boldsymbol{*} B} \Delta^n \boldsymbol{*} B \arrow[r] \arrow[d,hookrightarrow] & \Ca \\
          \Delta^n \boldsymbol{*} B^{\triangleright} \arrow[ur,dotted] & {}
          \end{tikzcd}
      \end{equation}
      This exists again of the fact that $B \to B^{\triangleright}$ is left anodyne and it follows again from \cite[Lemma 2.1.2.3]{HTT}.
      \end{enumerate}

 \end{proof}

 \begin{proposition}\label{finitelimitsinovercategories}
     Let $f : \Ca \to \Ca'$ be a functor between $\infty$-categories. Suppose that $\Ca$ admits pullbacks and the pullbacks are preserved by $f$. Then for any object $c \in \Ca$, $\Ca_{/c}$ admits finite limits and the limits are preserved by the functor  $f' : \Ca_{/c} \to \Ca'_{/f(c)}$.
     \end{proposition}

\begin{proof}
    We recall that taking pullbacks is taking limits for the diagram $\Lambda^2_2$ which is a weakly contractible simplicial set. Thus applying \cref{overcategorylimitpreserving}, we see that $\Ca_{/c}$ admits pullbacks and they are preserved by $\Ca_{/c} \to \Ca$. 

    We claim that the functor $f'$ preserves pullbacks. At first we notice the composition $\Ca_{/c} \to \Ca'_{/f(c)} \to \Ca$ which is same as $\Ca_{/c} \to \Ca \to \Ca'$ preserves pullbacks as it is composition of two such maps. Applying \cref{overcategorylimitpreserving} to the map $\Ca'_{/f(c)} \to \Ca'$ we see that the morphism $f'$ preserves pullbacks. \\

    Secondly $\Ca_{/c}$ admits pullbacks and has an final object, then by dual of \cite[Corollary 4.4.2.4]{HTT}, we see that $\Ca_{/c}$ admits finite limits. Thus we have $\Ca_{/c}$ admits finite limits and the map $f'$ sends the final objects $\op{id}_c$ to $\op{id}_{f(c)}$. Then by dual version of \cite[Corollary 4/4/2/5]{HTT}, we see that $f'$ preserves finite limits. 
\end{proof}
 \section{On $k$-truncated morphisms.}

 \begin{definition}\cite[Lemma 5.5.6.15]{HTT}
 Let $\Ca$ be an $\infty$-category admitting finite limits.
     A morphism $f : x \to y$ in an $\infty$-category is said to be $-2$-truncated if it is an equivalence. For $n \ge -1$ a morphism is $n$-truncated if the diagonal map $x \to x \times_y x$ is $n-1$-truncated.
 \end{definition}

 \begin{remark}
     The set of $n$-truncated morphisms are stable under pullbacks and compositions. Also given any $2$-simplex of $\Ca$ with edges opposite to vertex $1$ and $0$ are n-truncated, then the remaining edge is also $n$-truncated. Hence the class of $n$-trucated morphisms are admissible.\\
     The above definition also works if we assume that $\Ca$ admits pullbacks. Hence this works in our setup.
 \end{remark}

 \begin{lemma} \label{truncationdecreased}
    Let $\Ca$ be an $\infty$-category admitting pullbacks. Consider the following  square 
    \begin{equation}
        \begin{tikzcd}
            x \arrow[r,"f"] \arrow[d,"g"] & y \arrow[d,"h"] \\
            z \arrow[r,"p'"] & w
        \end{tikzcd}
    \end{equation}
    which admits a following decomposition
\begin{equation}
    \begin{tikzcd}
        x \arrow[r,"f"] \arrow[d,"g"] & y \arrow[d,"h'"] \arrow[ddr,bend left =60,"h"] & {} \\
        z \arrow[r,"p'"] \arrow[drr,"p", bend right =60] & w' \arrow[dr,"q"] & {} \\
        {} & {} & w
    \end{tikzcd}
\end{equation}
where the inside square is a pullback  square. If $q$ is $n$-truncated for $n \ge -1$, then the canonical map
\begin{equation}
   x:= z \times_{w'} y \to z \times_w y 
\end{equation}
is $n-1$-truncated.
 \end{lemma}

 \begin{proof}
We have the following commutative diagram :
\begin{equation}\label{cartdiagrampullbacksquare}
    \begin{tikzcd}
         x= z\times_{w'} y \arrow[r] \arrow[d] & y \arrow[r] \arrow[d] & w' \arrow[d]\\
         z \times_w y \arrow[r] & y \times_w w'\arrow[r] & w' \times_{w} w'
    \end{tikzcd}
\end{equation}
The first square is pullback square becuase : 
\begin{equation}
    y \times_w z \times_{y \times_w w'}y =z \times_{w'} (w' \times_w y) \times_{w' \times_w y} y = z \times_{w'} y.
\end{equation}
The second square is a pullback square because :
\begin{equation}
    y \times_w w' \times_{w' \times_w w'} w' = y \times_{w'} (w' \times_w w') \times_{w' \times_w w'} w' =  y \times_{w'} w' = y.
 \end{equation}
 As $q$ is $n$-truncted, $w' \times w' \times_w w'$ is $n-1$-truncated. As truncated morphisms are stable under pullback squares, we see that $x \to z \times_w y$ is $n-1$-truncated.
 \end{proof}

 \end{appendices}